\documentclass[11pt]{article}

\usepackage{times,amssymb,amsmath,exscale,array,latexsym}
\usepackage{graphicx}
\usepackage{epsfig}
\usepackage{mathrsfs}
\usepackage{mathtools}
\usepackage{stmaryrd}

\usepackage{amsthm}

\usepackage{color}
\definecolor{marin}{rgb}   {0.,   0.3,   0.7}
\definecolor{rouge}{rgb}   {0.8,   0.,   0.}
\definecolor{sepia}{rgb}   {0.8,   0.5,   0.}
\usepackage[colorlinks,citecolor=marin,linkcolor=rouge,
            bookmarksopen,
            bookmarksnumbered
           ]{hyperref}

\newcommand{\e}{\ensuremath{\mathrm{e}}}

\newcommand{\ts}{h}

\newtheorem{lemma}{Lemma}[section]

\newtheorem{theorem}[lemma]{Theorem}
\newtheorem{proposition}[lemma]{Proposition}
\newtheorem{corollary}[lemma]{Corollary}
\newtheorem{remark}[lemma]{Remark}

\numberwithin{equation}{section}

\newcommand{\QED}{\mbox{}\hfill \raisebox{-0.2pt}{\rule{5.6pt}{6pt}\rule{0pt}{0pt}}
          \medskip\par}

\title{On alternating-conjugate splitting methods} 

\author{J. Bernier$^{1}$, S. Blanes$^{2}$, F. Casas$^{3}$, A. Escorihuela-Tom\`as$^{4}$ \\[2ex]
$^{1}$ {\small\it Nantes Universit\'e, CNRS, Laboratoire de Math\'ematiques Jean Leray, LMJL, F-44000 Nantes, France}\\{
\small\it email: joackim.bernier@univ-nantes.fr}\\[1ex]
$^{2}$ {\small\it Universitat Polit\`ecnica de Val\`encia, Instituto de Matem\'atica Multidisciplinar, 46022-Valencia, Spain}\\{
\small\it email: serblaza@imm.upv.es}\\[1ex]
$^{3}$ {\small\it Departament de Matem\`atiques and IMAC, Universitat Jaume I, 12071-Castell\'on, Spain}\\{
\small\it email: Fernando.Casas@mat.uji.es}\\[1ex]
$^{4}$ {\small\it Departament de Matem\`atiques and IMAC, Universitat Jaume I, 12071-Castell\'on, Spain}\\{
\small\it email: alescori@uji.es}\\[1ex]
}

\begin{document}
\mathsurround 0.8mm

\maketitle

\begin{abstract}

The new  class of alternating-conjugate splitting methods is presented and analyzed. They are obtained by concatenating a given composition involving
complex coefficients with the same composition but with the complex conjugate coefficients.  
We show that schemes of this type exhibit good long
time behavior when applied to linear unitary and linear Hamiltonian systems, in contrast with other methods based on complex coefficients, and study in detail
their preservation properties. We also
present new schemes within this class up to order 6 that exhibit better efficiency than state-of-the-art splitting methods with real coefficients for some classes of problems.

\end{abstract}

\section{Introduction} \label{sec.1}

Splitting methods are widely used for the numerical integration in time of initial value problems
defined by ordinary and partial differential equations that can 
 be subdivided into different parts easier to solve than the original
system \cite{mclachlan02sm}. Splitting methods
 are typically easy to implement, present favorable properties concerning error propagation and  preserve in addition
a wide variety of structural properties inherent to the continuous system they are approximating. 
In this respect, they constitute paradigmatic examples of
geometric numerical integrators \cite{blanes16aci,hairer06gni}. Relevant examples are Hamiltonian systems separated into kinetic and potential energy in classical mechanics
and the time-dependent Schr\"odinger equation in quantum mechanics, where high-order splitting schemes have shown their effectiveness (see \cite{blanes24smf} for a recent review).

Advantageous as they are, splitting methods have also some drawbacks. One of them is related with an essential feature of this class of schemes: 
they necessarily involve negative coefficients when their order is three or higher. This makes them unsuitable, in particular, for non-reversible systems such as reaction-diffusion equations,
since the presence of negative coefficients induce severe instabilities in the numerical solution. 
Even in the context of ordinary differential equations, having negative coefficients usually leads to large truncation errors, and so it is recommended to include additional (free)
parameters to reduce these errors.

Splitting methods of order higher than two can still be used in
 equations evolved by semigroups, such as parabolic equations of evolution, if they involve 
instead \emph{complex} coefficients having \emph{positive real part}, as first shown in \cite{hansen09hos,castella09smw}. 
In that way not only the stability is recovered but in addition splitting methods of higher orders typically improve the efficiency with respect to lower orders over a wide
range of accuracies \cite{blanes13oho}. 

In fact, methods of this class were already
considered in the first explorations of splitting methods in the early 1990s \cite{bandrauk91ies,suzuki90fdo,suzuki91gto,suzuki95hep} and later 
in the context of symplectic integration of 
Hamiltonian systems \cite{chambers03siw} and quantum mechanics  \cite{bandrauk06cis,prosen06hon}. This was somehow natural, since it was soon realized
that the order conditions, even of low order, admit complex solutions in addition to the real ones.

More recently, a systematic study of splitting methods with
complex coefficients has been carried out, not only by designing new high order schemes but also by analyzing their preservation properties, for both ordinary and partial
differential equations \cite{blanes22asm,blanes22osc,bernier23scs,blanes24scs}. As is well known, when considering splitting methods with real coefficients, 
left-right palindromic (or symmetric) compositions are usually preferred, since their construction is easier and
in addition they provide time-symmetric approximations to the exact solution. The situation is different, however, with complex coefficients, at least in the context
of linear equations evolving in the unitary group \cite{blanes22osc,bernier23scs} and also for parabolic problems \cite{blanes24scs}. In fact, when considering
complex coefficients, other patterns are more suitable. In particular,  
\emph{symmetric-conjugate} schemes (i.e., methods with complex coefficients which are symmetric in the real
part and antisymmetric in the imaginary part) present good preservation properties in both cases. This is true even when dealing with real ODEs and the numerical
approximation is projected after each time step to 
its real part \cite{blanes22osc}.

Although the results presented in the above mentioned references concerning symmetric-conjugate splitting methods
are promising for their practical application to integrate linear unitary problems
and evolution equations of parabolic type, it is fair to say that the analysis of splitting methods involving complex coefficients is far from complete. In particular
several questions naturally arise: are there other types of schemes also possessing good preservation properties when applied to this class of problems? 
In the case of real Hamiltonian systems, is the energy almost preserved for long time integrations, as for splitting methods with real coefficients?
 What happens when one applies these methods to nonlinear evolution problems, such as nonlinear
Schr\"odinger or complex Ginzburg--Landau equations?

In this work we address some of these issues, and in particular we present and analyze a rather general class of splitting methods with complex coefficients
obtained by concatenating a given composition with the same composition but with complex conjugate coefficients. We show that schemes of this type (which can be properly called  \emph{alternating-conjugate} methods)
exhibit a good long time behavior when they
are applied to linear unitary and Hermitian flows with a sufficiently small time step, and also for linear Hamiltonian systems, due to their
special spectral properties. This is in contrast with symmetric and symmetric-conjugate schemes and account for some observations previously reported in the literature.

The plan of the paper is as follows. In section \ref{sec.2} we carry out some numerical experiments illustrating the behavior exhibited by different splitting methods
of order 3 and 4, and in particular the good preservation properties of alternating-conjugate schemes. 
Section \ref{sec.3} is devoted to analyze the qualitative properties
of these methods applied to linear unitary and Hermitian flows, whereas section \ref{sec.4} deals with linear real Hamiltonian systems. Section \ref{sec.5}
presents an strategy to construct alternating-conjugate composition schemes based on the study of the modified vector field. The new methods are then
illustrated on several numerical examples. 
Finally, section \ref{sec.6} contains some concluding remarks.

\section{Some numerical observations}
\label{sec.2}

In this section we consider the simple problem of approximating $U(t) = \e^{i t H}$, where $H$ is a $N \times N$ Hermitian matrix. As is well known,
$U(t)$ is a unitary matrix, $U^{*}(t) = U(t)^{-1} \in \mathbb{C}^N \times \mathbb{C}^N$, and can be seen as the solution of the differential equation

\begin{equation} \label{uf.1}
   \frac{d U}{dt} = i H U, \qquad U(0) = I. 
\end{equation}
Suppose $H$ can be decomposed as $H = A + B$ so that computing $\e^{i t A}$ and $\e^{i t B}$ is easier (or less costly) than solving \eqref{uf.1}. Then it makes sense 
to construct splitting methods of the form 
\begin{equation} \label{bch1}
 \Psi_{h}^{[p]} = \e^{i a_{s+1} h A} \, \e^{i b_s h B} \, \e^{i a_s h A} \, \cdots \, \e^{i a_2 h A} \, \e^{i b_1 h B} \, \e^{i a_{1} h A}
\end{equation}
with suitable coefficients $a_j$, $b_j$ so that
\[
  U(h) = \Psi_{h}^{[p]} + \mathcal{O}(h^{p+1}), \qquad h \rightarrow 0
\]  
for a time step $t = h$. Simple examples are the Lie--Trotter scheme $\Psi_{h}^{[1]} = \e^{i h A} \, \e^{i h B}$ and the Strang splitting
\begin{equation} \label{strang}
  \Psi_{h}^{[2]} \equiv \e^{\frac{1}{2} i h A} \, \e^{i h B} \, \e^{\frac{1}{2} i h A},  
\end{equation}
of order $p=1$ and 2, respectively. Methods with complex coefficients already exist at order 3 \cite{bandrauk91ies},
\begin{equation} \label{scheme3}
  \Psi_{h}^{[3]} = \e^{i a_1 h A} \, \e^{i b_1 h B} \, \e^{i a_2 h A} \, \e^{i \overline{b}_1 h B} \, \e^{i \overline{a}_1 h A}, \qquad 
  a_1 = \frac{\gamma}{2}, \;\; b_1 = \gamma = \frac{1}{2}  \pm i \frac{\sqrt{3}}{6}, \;\; a_2 = \frac{1}{2},
\end{equation}
which corresponds to the composition of the Strang splitting 
\begin{equation} \label{bc3}
   \Psi_{h}^{[3]}  = \Psi_{\gamma h}^{[2]} \,\, \Psi_{\overline{\gamma} h}^{[2]}
\end{equation} 
and order 4, namely
\begin{equation}   \label{schemes4}
\begin{aligned}
 & \Psi_{h,p}^{[4]} = \e^{i a_1 h A} \, \e^{i b_1 h B} \, \e^{i a_2 h A} \, \e^{i b_2 h B} \, \e^{i a_2 h A} \, \e^{i b_1 h B} \, \e^{i a_1 h A}   \\
 & \Psi_{h,sc}^{[4]} =  \e^{i a_1 h A} \, \e^{i b_1 h B} \, \e^{i a_2 h A} \, \e^{i b_2 h B} \, \e^{i \overline{a}_2 h A} \, \e^{i \overline{b}_1 h B} \, \e^{i \overline{a}_1 h A}. 
\end{aligned} 
\end{equation}
The coefficents $a_j, b_j$ in \eqref{schemes4} are obtained by inserting the Strang splitting \eqref{strang} into the triple-jump composition \cite{creutz89hoh,suzuki90fdo,yoshida90coh}
\begin{equation} \label{tjc}
  \Psi_{\gamma_1 h}^{[2]} \,\, \Psi_{\gamma_2 h}^{[2]} \,\, \Psi_{\gamma_3 h}^{[2]}
\end{equation}
with 
\[
\begin{aligned}
 & \Psi_{h,p}^{[4]}:     \qquad       \gamma_1=\gamma_3 = \frac{1}{2 - 2^{1/3} \e^{2 i \ell \pi/3}}
	, \qquad
	\gamma_2 = 1- 2 \gamma_1
	, \qquad  \ell=1,2 \\
 & \Psi_{h,sc}^{[4]}:   \qquad	  \gamma_1=\overline{\gamma}_3=  \frac{1}4 \pm i \, \frac14\sqrt{\frac53}
	,  	\qquad \gamma_2 = \frac12, 
\end{aligned}
\]
respectively. Notice the different pattern in the distribution of the coefficients: whereas $\Psi_{h,p}^{[4]}$ is a \emph{palindromic} 
(and hence time-symmetric or time-reversible) composition, both $\Psi_{h}^{[3]}$ 
and $\Psi_{h,sc}^{[4]}$ are
\emph{symmetric-conjugate} \cite{blanes22osc}. 

For our first experiment we take $N=10$, generate a \emph{real} symmetric matrix $H$ with uniformly distributed elements in $(0,1)$ so that it has simple eigenvalues, 
and then decompose $H$ as a sum of (generally non-symmetric) \emph{real} matrices
$A$ and $B$. Specifically, we generate $A$ as a real matrix with uniformly distributed elements in $(0,1)$ and then $B$ is taken as $B = H - A$.
Then we compute the matrices generated by $\Psi_{h}^{[3]}$, $\Psi_{h,p}^{[4]}$ and $\Psi_{h,sc}^{[4]}$ for different values of $h$, determine their eigenvalues $\omega_j$
and compute
\begin{equation} \label{dh}
  D_h = \max_{1 \le j \le N} (|\omega_j| -1).
\end{equation}  
This quantity is finally plotted as a function of $h$. The results are shown in Figure \ref{figure.1} (left): both $\Psi_{h}^{[3]}$ (black) and $\Psi_{h,sc}^{[4]}$ (blue) behave as unitary maps for values of
$h$ in an interval $(0,h^*)$ for some $h^*$ depending on the particular $H$ taken ($|\omega_j| = 1$ for all $j$, except round-off), but this is not the case of the time-symmetric
composition $\Psi_{h,p}^{[4]}$ (red). The plot illustrates a feature rigorously established in our previous contribution \cite{bernier23scs}: symmetric-conjugate splitting methods,
under quite general hypothesis on the matrix $H$, are indeed conjugate to unitary maps for sufficiently small values of $h$, and in addition, they preserve the norm and energy of the
numerical solution over long time integrations.

We next consider the same problem but now we alternate each method with the one obtained by reversing the sign of the imaginary part of the coefficients. 
In other words, we apply the 4th-order alternating-conjugate splitting schemes
\begin{equation} \label{ac.1}
  S_{h}^{(1)} \equiv \Psi_{h/2}^{[3]} \,\,  \widetilde{\Psi}_{h/2}^{[3]}, \qquad\quad S_{h}^{(2)} \equiv \Psi_{h/2,p}^{[4]} \,\,  \widetilde{\Psi}_{h/2,p}^{[4]} \qquad 
  \mbox{ and } \qquad 
  S_{h}^{(3)} \equiv \Psi_{h/2,sc}^{[4]} \,\, \widetilde{\Psi}_{h/2,sc}^{[4]}
\end{equation}
with
\[
\begin{aligned}
  & \widetilde{\Psi}_{h}^{[3]} = \e^{i \, \overline{a}_1 h A} \, \e^{i \, \overline{b}_1 h B} \, \e^{i a_2 h A} \, \e^{i b_1 h B} \, \e^{i a_1 h A} \\
 & \widetilde{\Psi}_{h,p}^{[4]} = \e^{i \, \overline{a}_1 h A} \, \e^{i \, \overline{b}_1 h B} \, \e^{i \, \overline{a}_2 h A} \, \e^{i \, \overline{b}_2 h B} \, 
 \e^{i \, \overline{a}_2 h A} \, \e^{i \, \overline{b}_1 h B} \, \e^{i \, \overline{a}_1 h A}   \\
 & \widetilde{\Psi}_{h,sc}^{[4]} =  \e^{i \, \overline{a}_1 h A} \, \e^{i \, \overline{b}_1 h B} \, \e^{i \, \overline{a}_2 h A} \, \e^{i b_2 h B} \, \e^{i a_2 h A} \, \e^{i b_1 h B} \, 
 \e^{i a_1 h A} 
\end{aligned} 
\]
and compute again $D_h$ as a function of $h$. Notice that $S_{h}^{(2)}$ is now symmetric-conjugate, whereas both $S_{h}^{(1)}$ and $S_{h}^{(3)}$ 
are time-symmetric. In spite of this, all of them
behave as unitary maps in a certain interval of values of $h$, as shown in Figure \ref{figure.1} (right). Thus, by concatenating a symmetric method $\Psi_{h}$ with
$\widetilde{\Psi}_{h}$
we get a new scheme of the same (or higher) 
order and better preservation properties, whereas alternating a symmetric-conjugate method does not degrade its good behavior. Notice that
for this example $\widetilde{\Psi}_{h}$ is just the complex conjugate of $\Psi_{h}$, since $A$ and $B$ are real.

\begin{figure}[!ht] 
\centering
  \includegraphics[width=.49\textwidth]{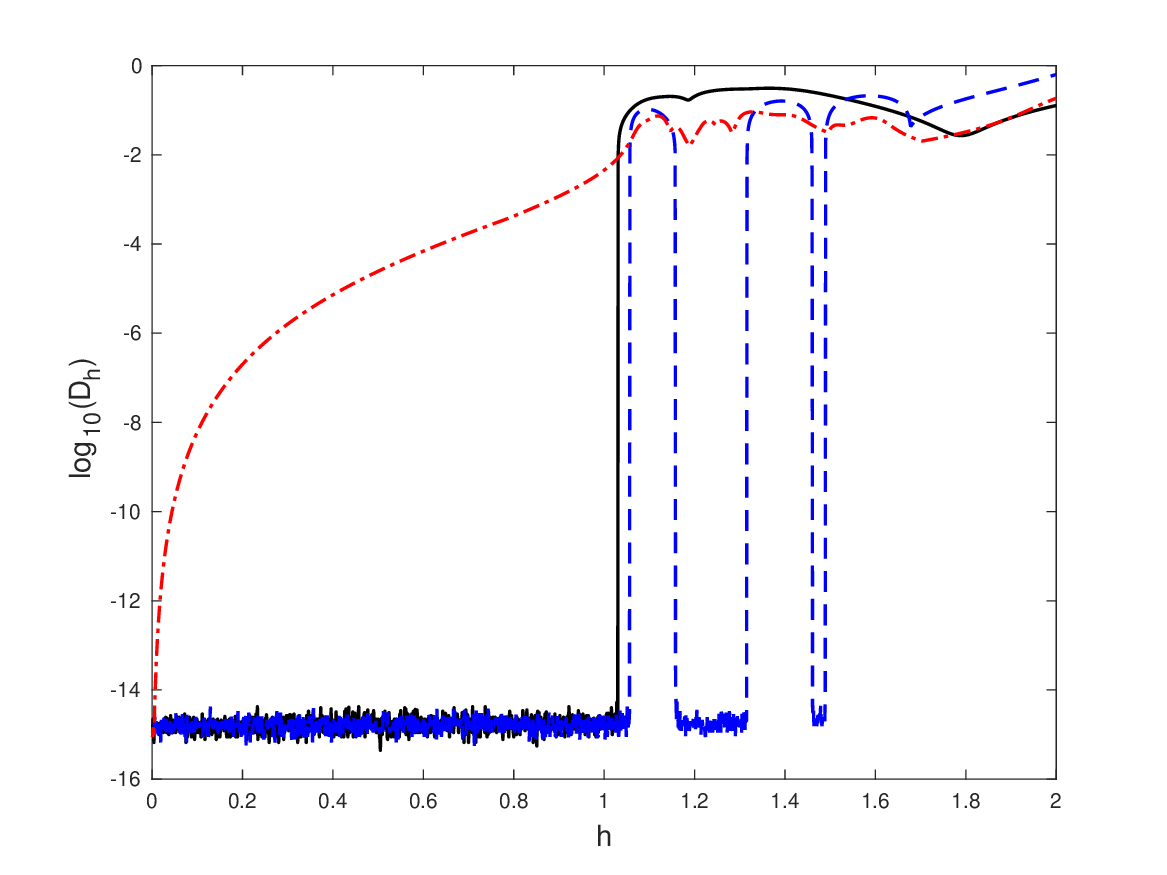}
  \includegraphics[width=.49\textwidth]{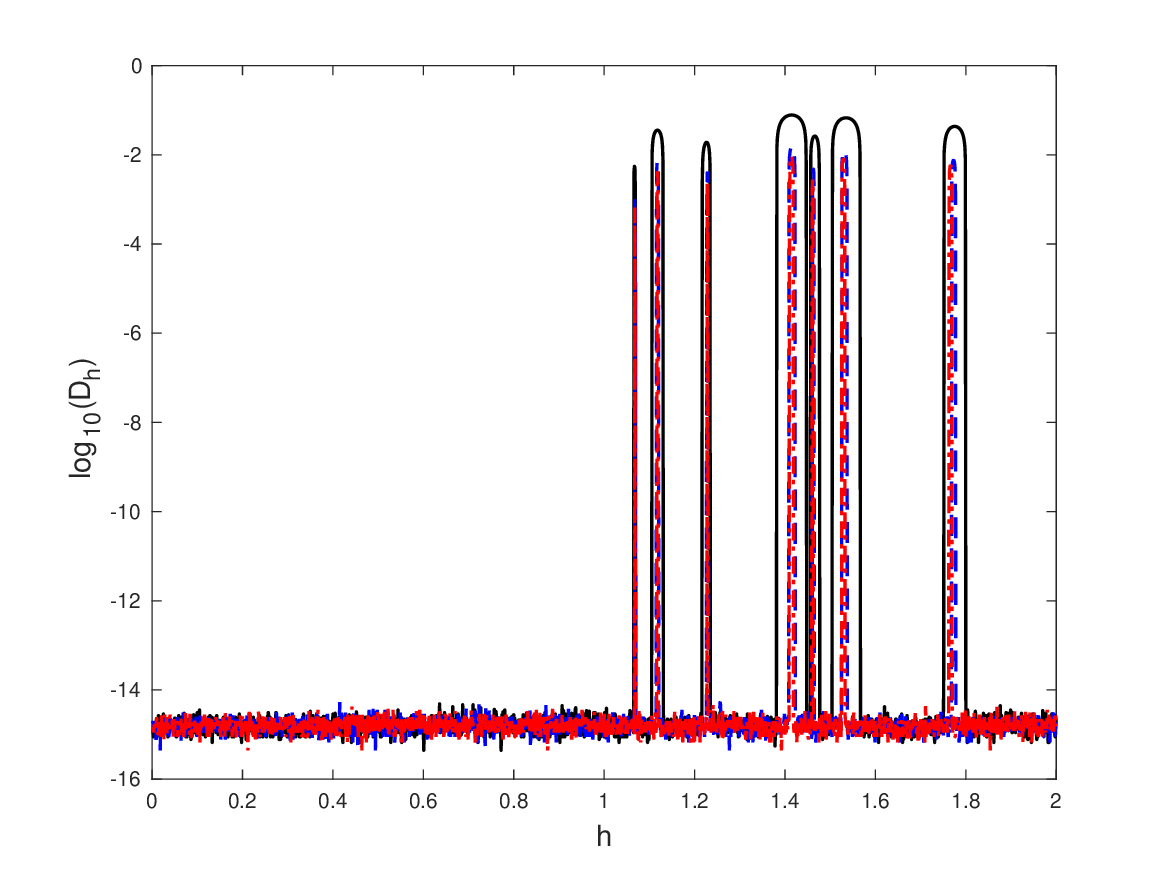}
\caption{\label{figure.1} \small Left panel: function $D_h$ defined by \eqref{dh} vs. the step size $h$ obtained by schemes \eqref{scheme3} and \eqref{schemes4} applied
to eq. \eqref{uf.1} when $H$ is a $10 \times 10$ \emph{real symmetric} matrix with simple eigenvalues. 
$\Psi_{h}^{[3]}$ (black {solid}) and $\Psi_{h,sc}^{[4]}$ (blue {dashed}) are symmetric-conjugate, whereas
$\Psi_{h,p}^{[4]}$ (red {dash-dotted}) is palindromic. Right panel: results achieved by the alternating-conjugate splitting methods \eqref{ac.1} on the same problem. Now all methods
behave as unitary maps for small enough $h$.}
\end{figure}

In our second experiment we generate a random Hermitian matrix $H \in \mathbb{C}^N \times \mathbb{C}^N$ with \emph{simple} 
eigenvalues, split $H$ as a sum of two 
\emph{complex} matrices $A$, $B$, and compare again the results obtained
by the previous schemes. {Here again we form a complex matrix $A_1$ with real and imaginary parts generated from the uniform distribution in $(0,1)$, then we
take $A = A_1 + A_1^*$ and finally $B= H-A$.}
Figure \ref{figure.2} (left) corresponds to methods \eqref{scheme3} and \eqref{schemes4}. 
Now $D_h > 0$ for any $h > 0$ for both the palindromic
and the symmetric-conjugate 4th-order splitting methods  \eqref{schemes4}, whereas $\Psi_{h}^{[3]}$ still behaves as a unitary method for $h < h^*$. 
If, on the other hand, the alternating-conjugate compositions \eqref{ac.1} are considered, then
the resulting approximations have all their eigenvalues $|\omega_j| = 1$ for small enough $h$(Figure \ref{figure.2}, right). 
We thus may conclude that concatenating a given method with the same composition
with complex conjugate coefficients leads to a new scheme with improved preservation properties, at least for unitary flows. One may wonder why scheme \eqref{scheme3} still behaves
as a unitary map here: the reason is that it is itself a composition of this type, as \eqref{bc3} clearly shows.

\begin{figure}[!ht] 
\centering
  \includegraphics[width=.49\textwidth]{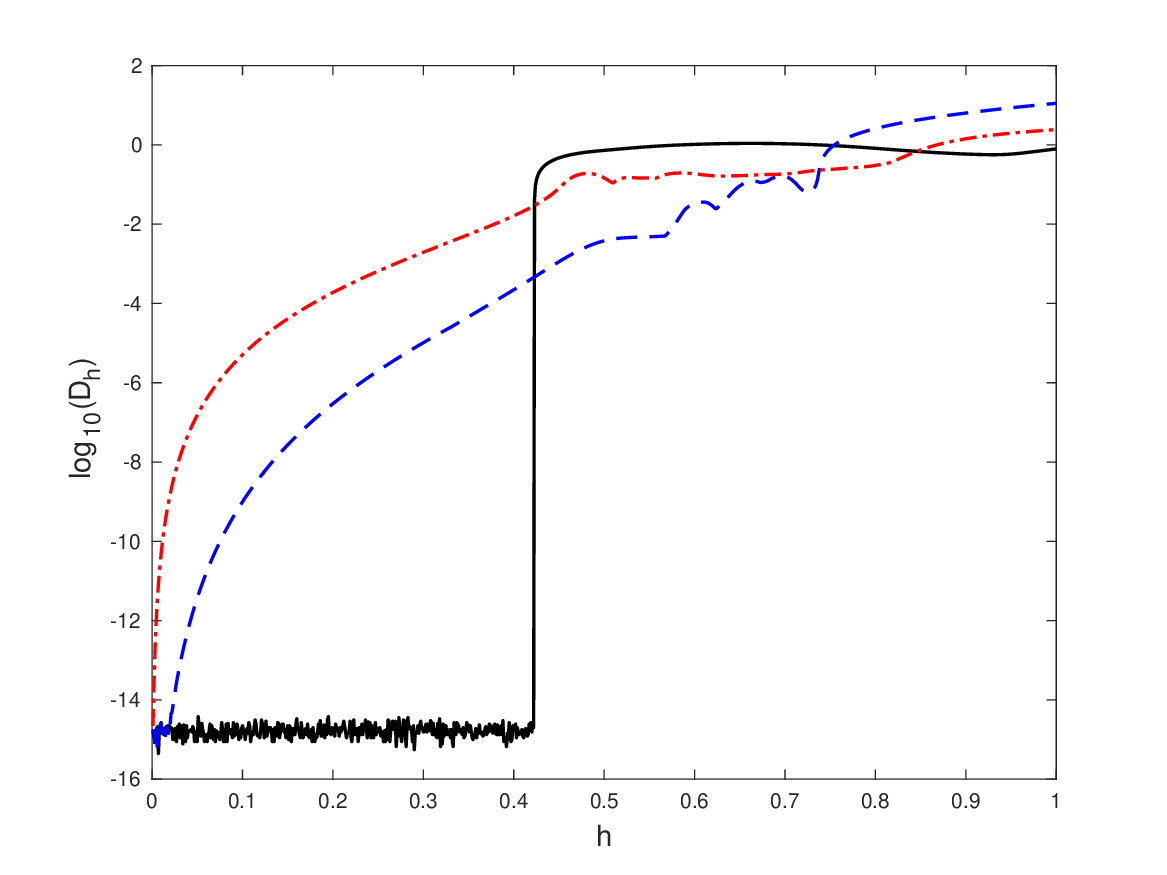}
  \includegraphics[width=.49\textwidth]{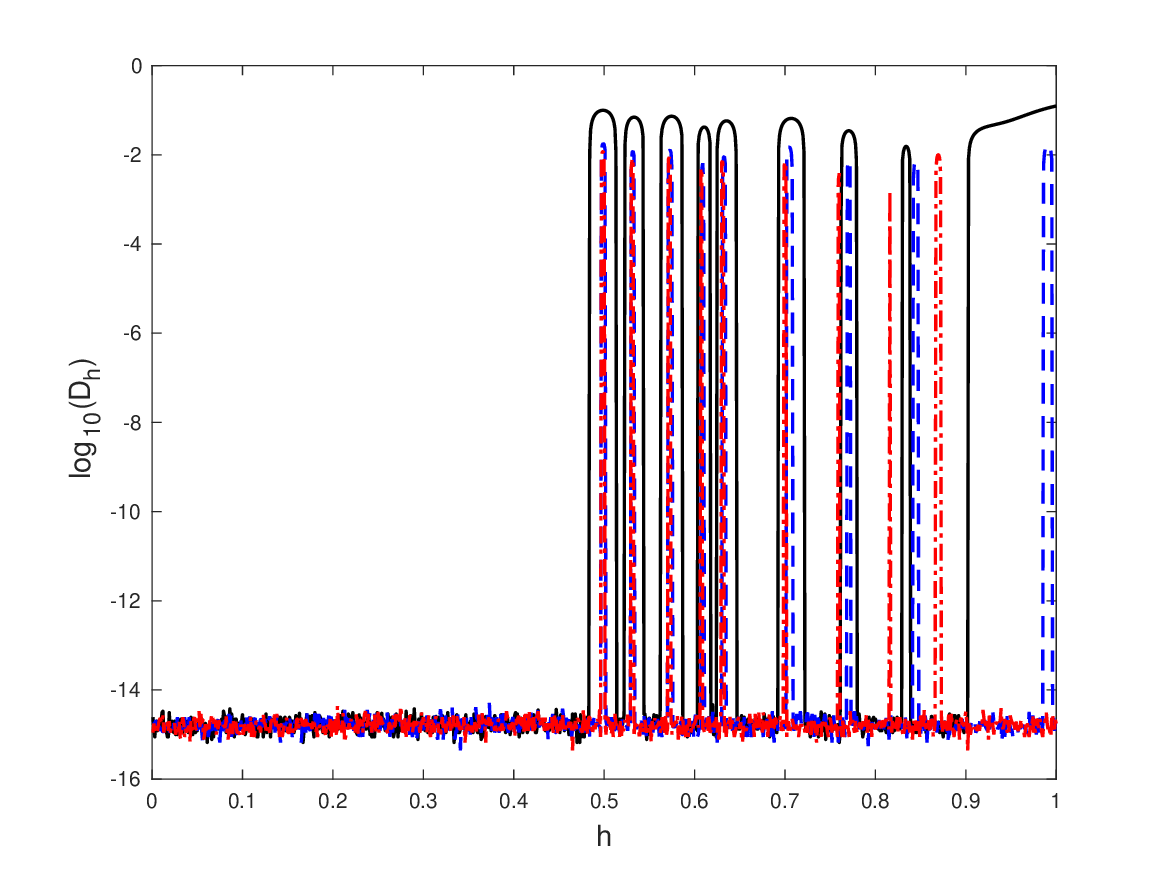}
\caption{\label{figure.2} \small Same as Figure \ref{figure.1} when $H$ is a $10 \times 10$ \emph{complex} Hermitian matrix with \emph{simple} eigenvalues. 
The schemes depicted are $\Psi_{h}^{[3]}$ (black {solid}), $\Psi_{h,sc}^{[4]}$ (blue {dashed}) and
$\Psi_{h,p}^{[4]}$ (red {dash-dotted}). Right panel corresponds to the alternating-conjugate version of the previous methods \eqref{ac.1}. In this case, symmetric-conjugate
methods do not provide a correct qualitative description of the system.
}
\end{figure}

Finally, we take again a Hermitian matrix $H \in \mathbb{C}^N \times \mathbb{C}^N$, but now with \emph{repeated} eigenvalues. {This is done by first fixing a diagonal matrix $C$ with the eigenvalues, then by generating a random unitary matrix $P$ and finally computing $H = P C P^*$. Then we proceed as in the first example
to express $H$ as the sum of two Hermitian matrices $A$ and $B$}. 
 The results for \eqref{scheme3} and \eqref{schemes4} are depicted in Figure \ref{figure.3} (left), whereas those for the
alternating-conjugate schemes \eqref{ac.1} are shown in Figure \ref{figure.3} (right). Notice that, among splitting methods with complex coefficients,  
alternating-conjugate methods provide the best qualitative description of the unitary flow $U(t) = \e^{i t H}$. 

\begin{figure}[!ht] 
\centering
  \includegraphics[width=.49\textwidth]{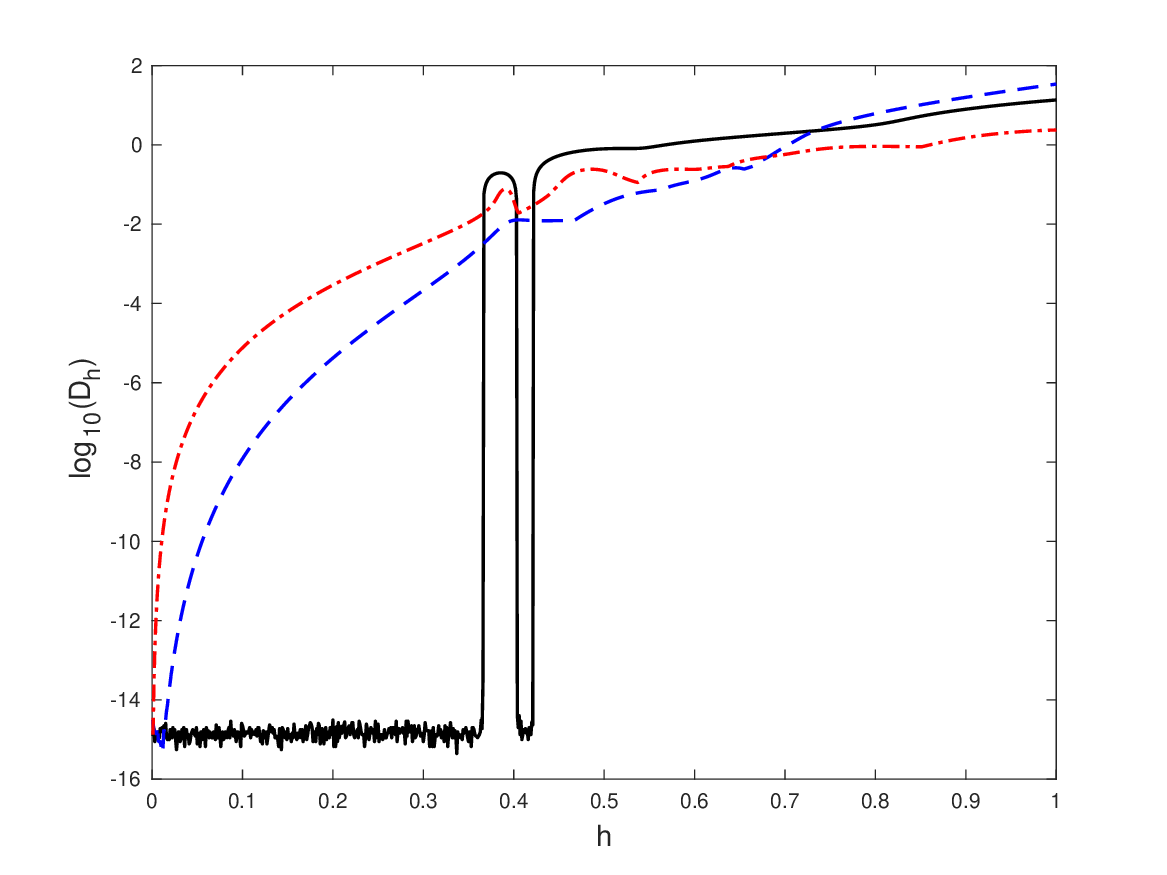}
  \includegraphics[width=.49\textwidth]{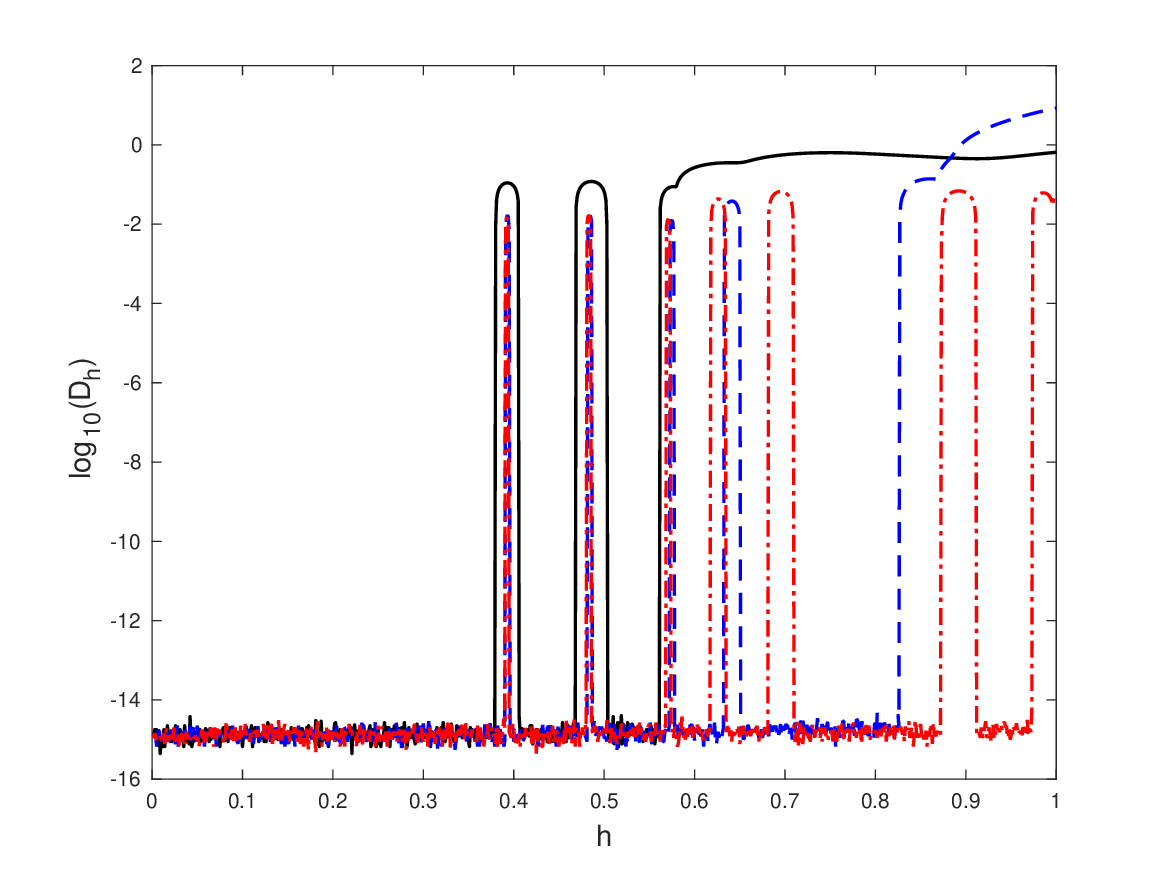}
\caption{\label{figure.3} \small Same as Figures \ref{figure.1} and \ref{figure.2} when $H$ is a $10 \times 10$ \emph{complex} Hermitian matrix with \emph{repeated}
 eigenvalues. The schemes depicted are $\Psi_{h}^{[3]}$ (black {solid}), $\Psi_{h,sc}^{[4]}$ (blue {dashed}) and
$\Psi_{h,p}^{[4]}$ (red {dash-dotted}). Right panel corresponds to the alternating-conjugate version of the previous methods \eqref{ac.1}.
Only alternating-conjugate methods behave as unitary integrators in this case. 
}
\end{figure}

While the observed behavior of
symmetric-conjugate schemes can be understood in virtue of the study carried out in \cite{bernier23scs}, this analysis does not explain the favorable properties exhibited by 
alternating-conjugate methods in the above examples, so that an extension is clearly desirable. This is addressed in the next section, largely based on the
spectral properties of the schemes.

\section{Alternating-conjugate splitting methods and their preservation properties }
\label{sec.3}

For our research we consider the generic composition
\begin{equation}
\label{eq:Sh}
S_h := \e^{\ts\alpha_1 A_1} \, \cdots \, \e^{\ts\alpha_m A_m} \, \e^{\ts\overline{\alpha}_1 A_1} \, \cdots \, \e^{ \ts \overline{\alpha}_m A_m},
\end{equation}
where $\ts \ll 1$ is the time step, $m\geq 2$, $A_1,\ldots, A_m \in \mathbb{C}^{N\times N}$ are some $N\times N$ complex matrices and $\alpha_1,\cdots,\alpha_m \in \mathbb{C}$ are certain coefficients.

Clearly, scheme \eqref{eq:Sh} can be expressed as
\begin{equation}
\label{eq:psihtilde}
S_h = \Psi_{\ts} \, \widetilde{\Psi}_{\ts}, \qquad \mbox{ where } \qquad \left\{ \begin{aligned}
  & \Psi_{\ts} =  \e^{\ts\alpha_1 A_1} \, \cdots \, \e^{\ts\alpha_m A_m} \\
   & \widetilde{\Psi}_{\ts} =  \e^{\ts\overline{\alpha}_1 A_1} \, \cdots \, \e^{\ts\overline{\alpha}_m A_m}, \end{aligned} \right.
\end{equation}
and our goal is to show that $S_h$ actually enjoys good qualitative properties, in the sense that it preserves some of the spectral symmetries of the equations it approximates.

It is worth remarking here that, among methods \eqref{eq:Sh}, some of them possess extra symmetries and are of special importance. 
Indeed, if $\Psi_\ts$ is a symmetric-conjugate method, i.e. 
\begin{equation}
\label{eq:sym_conj}
\alpha_{n-j+1} = \overline{\alpha}_{j}  \qquad \mbox{ and } \qquad A_{n-j+1} = A_j,
\end{equation}
then $\widetilde{\Psi}_{\ts}$ is also symmetric-conjugate, and the full method $S_h$ is palindromic, so that $S_h$ is time-reversible, 
$S_h^{-1} = S_{-h}$. Particular examples are schemes $S_h^{(1)}$ and $S_h^{(3)}$ in \eqref{ac.1}.

\subsection{Spectral symmetries}
\label{s_spectral}

For the analysis it is convenient to discriminate among several types of matrices and establish the spectral properties possessed by the method in each case.

\paragraph{Real matrices.} Suppose first that the matrices $A_1,\ldots, A_m \in \mathbb{R}^{N\times N}$ are real. Therefore, $S_\ts$ is consistent with the flow of a real matrix, i.e. there exists a real matrix $M\in \mathbb{R}^{N\times N}$ such that
$$
S_\ts = \e^{h M} + \mathcal{O}(h^2).
$$
Since $M$ is real, the spectrum of its flow is invariant by complex conjugation, i.e. $\sigma(\e^{h M}) = \overline{\sigma(\e^{h M})}$. In that case 
the spectrum of $S_h$ has the same symmetry, as the following argument shows.

As noted before, since the matrices  $A_1,\ldots, A_m \in \mathbb{R}^{N\times N}$ are real, then $\widetilde{\Psi}_{\ts}$ and  $\Psi_{\ts}$ are complex conjugated, 
$\widetilde{\Psi}_{\ts} = \overline{\Psi_{\ts}}$. Therefore, $S_h$ is similar to $\overline{S_h}$:
$$
\overline{S_h} = \overline{\Psi_{\ts} \, \overline{\Psi_{\ts}} } = \overline{\Psi_{\ts}} \, \Psi_{\ts} = \Psi_{\ts}^{-1} \, S_h  \,\Psi_{\ts},
$$
and a fortiori we deduce that the spectrum of $S_h$ is also invariant by complex conjugation:
$$
\sigma(S_\ts)  = \sigma(\overline{S_\ts}) =  \overline{\sigma(S_\ts)}.
$$

\paragraph{Skew-Hermitian matrices.} 

 Let us assume now that the matrices $A_1,\ldots, A_m \in \mathbb{C}^{N\times N}$ are skew-Hermitian. In consequence, $S_\ts$ is consistent with a unitary flow, i.e. there exists a Hermitian matrix $H\in \mathbb{C}^{N\times N}$ such that
$$
S_\ts = \e^{i h H} + \mathcal{O}(h^2).
$$
$H$ being Hermitian, its eigenvalues are real and so the spectrum satisfies
$\sigma(\e^{i h H}) = \overline{\mathrm{inv}( \sigma(\e^{i h H}) )}$, where $\mathrm{inv} : z\mapsto z^{-1}$ denotes the inverse function. 

It turns out that the same property holds for the spectrum of $S_h$: since the matrices  $A_1,\ldots, A_m \in \mathbb{C}^{N\times N}$ are 
skew-Hermitian, then $\widetilde{\Psi}_{\ts}$ is the inverse of the adjoint $\Psi_h^* \equiv \overline{\Psi_h}^T$, i.e.,
$$
 \widetilde{\Psi}_{\ts} = \Psi_{\ts}^{-*}.
$$
Therefore, we deduce that $S_h$ is similar to $S_h^{-*}$:
$$
S_h^{-*} = (\Psi_{\ts}\Psi_{\ts}^{-*})^{-*} = \Psi_{\ts}^{-*} \Psi_{\ts} = \Psi_{\ts}^{-1} S_\ts  \Psi_{\ts}
$$
and finally
$$
\sigma(S_\ts)  = \sigma(S_\ts^{-*}) = \overline{\mathrm{inv}(\sigma(S_\ts))},
$$
as the exact flow.

\paragraph{Hermitian matrices.} Suppose the matrices $A_1,\ldots, A_m \in \mathbb{C}^{N\times N}$ are Hermitian. Therefore, $S_\ts$ is consistent with a Hermitian flow, i.e. there exists a Hermitian matrix $H\in \mathbb{C}^{N\times N}$ such that
$$
S_\ts = e^{h H} + \mathcal{O}(h^2).
$$
Since $H$ is Hermitian, its eigenvalues are real and so $\sigma(\e^{ h H}) =  \overline{\sigma(\e^{ h H})}$. The property is also satisfied by $S_h$ when 
$\Psi_h$ is a symmetric-conjugate method (see \eqref{eq:sym_conj}).

Since the matrices  $A_1,\ldots, A_m \in \mathbb{C}^{N\times N}$ are Hermitian and $\Psi_{\ts}$, $\widetilde{\Psi}_{\ts}$ are symmetric-conjugated, 
then  $\Psi_{\ts}$ and $\widetilde{\Psi}_{\ts}$ are also Hermitian, i.e.
$$
\Psi_\ts^* = \Psi_\ts \qquad  \mbox{ and } \qquad \widetilde{\Psi}_{\ts}^* =  \widetilde{\Psi}_{\ts},
$$
and thus $S_h$ is similar to $S_h^{*}$,
$$
S_h^{*} = (\Psi_{\ts}\widetilde{\Psi}_{\ts})^{*} = \widetilde{\Psi}_{\ts}^* \Psi_\ts^*  = \widetilde{\Psi}_{\ts} \Psi_\ts = \Psi_{\ts}^{-1} S_\ts  \Psi_{\ts}.
$$
In consequence, the spectrum of $S_h$ is invariant by complex conjugation:
$$
\sigma(S_\ts)  = \sigma(S_\ts^{*}) =  \overline{\sigma(S_\ts)}.
$$

\subsection{Perturbation theory}
The spectral properties we have derived in the previous subsection are interesting but somehow quite weak. For example, when we aim at discretizing a unitary flow, we would like the numerical flow to be also unitary or at least to be similar (conjugate)
to a unitary flow. From the point of view of the eigenvalues, it means that we would not only want that they satisfy the symmetry $\sigma(S_\ts) = \overline{\mathrm{inv}(\sigma(S_\ts))}$ but they are of modulus one, that is $\sigma(S_\ts) \subset \mathbb{U}$, where $\mathbb{U} \subset \mathbb{C}$ denotes the unit circle.

We will see in later subsections that, by using perturbation theory arguments (in the same spirit as those used for reversible integrators, see e.g. 
\cite{bernier23scs}), these invariance properties can be typically strengthened to ensure that the long time behavior of the numerical flow is consistent with the one of the continuous flow. All these results rely on the following technical lemma which can be seen as a consequence of the theory of Birkhoff normal forms adapted to our setting.
It constitutes in fact a generalization of \cite[Lemma 3.7]{bernier23scs}.

\begin{lemma}
\label{lem:diag_cont}
Let $\mathfrak{g}$ be a  Lie algebra of matrices and $M_\ts \in \mathfrak{g}$ be a family of matrices depending smoothly on $\ts$ and of the form 
$$
M_\ts = M_0 + \mathcal{O}(\ts^p), \qquad \mbox{ where } \qquad  p\geq 1.
$$
 If $M_0$ is diagonalizable on $\mathbb{C}$ then there exists a family of matrices $\chi_\ts \in \mathfrak{g}$, depending smoothly on $\ts$, such that if $|\ts|$ is small enough, $\e^{-\ts^p\chi_\ts} M_\ts \, \e^{\ts^p\chi_\ts}$ commutes with $M_0$, i.e.
$$
[\e^{\ts^p\chi_\ts} M_\ts \, \e^{-\ts^p\chi_\ts} , M_0 ] =0.
$$
\end{lemma}
\begin{proof} We aim at designing the family $\chi_\ts \in \mathfrak{g}$ as solution of the equation
\begin{equation}
\label{eq:v0}
\mathrm{ad}_{M_0} \left( \e^{\ts^p\chi_\ts} M_\ts \, \e^{-\ts^p\chi_\ts} \right) =0.
\end{equation}
Let $N\geq 1$ be the integer such that $\mathfrak{g}\subset \mathfrak{gl}_N(\mathbb{C})$ (the Lie algebra of the $N\times N$ complex matrices). For a given matrix $A\in  \mathfrak{gl}_N(\mathbb{C})$, we define its adjoint representation $\mathrm{ad}_A :  \mathfrak{gl}_N(\mathbb{C}) \to  \mathfrak{gl}_N(\mathbb{C})$ by 
$$
\mathrm{ad}_A B := [A,B] = AB-BA.
$$
Thanks to the well known identity $\e^{A} B \, \e^{-A} = \e^{\mathrm{ad}_A} B$, equation \eqref{eq:v0} is just
\begin{equation} \label{eq:v1}
 \mathrm{ad}_{M_0} \left( \e^{\ts^p \mathrm{ad}_{\chi_\ts}}  M_\ts \right) = 0.
\end{equation}
By hypothesis, $M_h$ can be written as $M_\ts = M_0 + \ts^p R_\ts$,
where $R_\ts \in \mathfrak{g}$ is a family of matrices depending smoothly on $\ts$. {Now
\[
  \e^{h^p \mathrm{ad}_{\chi_\ts}} M_h = \e^{h^p \mathrm{ad}_{\chi_\ts}} M_0 + h^p \e^{h^p \mathrm{ad}_{\chi_\ts}} R_h
\]
and
\[ 
 \e^{h^p \mathrm{ad}_{\chi_\ts}} M_0 = \big( I +  \varphi_1(h^p \mathrm{ad}_{\chi_h}) \, h^p \mathrm{ad}_{\chi_h} \big) M_0 =
  M_0 - h^p  \varphi_1(h^p \mathrm{ad}_{\chi_h}) \mathrm{ad}_{M_0} \chi_h,
\]
where $\varphi_1(z) := \frac{e^z - 1}z$. In consequence,
\[
    \mathrm{ad}_{M_0} \left( \e^{\ts^p \mathrm{ad}_{\chi_\ts}}  M_\ts \right) = h^p \, \mathrm{ad}_{M_0} \left( 
    \e^{h^p \mathrm{ad}_{\chi_\ts}} R_h - \varphi_1(h^p \mathrm{ad}_{\chi_h}) \,  \mathrm{ad}_{M_0} \chi_h \right)
\]
}
so that equation \eqref{eq:v1} is equivalent to
$$
f(\ts,\chi_\ts) := \mathrm{ad}_{M_0} \left(   \e^{\ts^p \mathrm{ad}_{\chi_\ts}}  R_\ts-\varphi_1(\ts^p \mathrm{ad}_{\chi_{\ts}}) \, \mathrm{ad}_{M_0} \chi_\ts  \right) = 0.
$$
Since $\mathfrak{g}$ is a finite dimensional Lie algebra, 
we restrict ourselves to  $\chi_\ts$ in $ \mathrm{Im} \, (\mathrm{ad}_{M_0})_{| \mathfrak{g}}$ and consider $f$ as a 
smooth map from $\mathbb{R} \times  \mathrm{Im} \, (\mathrm{ad}_{M_0})_{| \mathfrak{g}}$ to $ \mathrm{Im} \, (\mathrm{ad}_{M_0})_{| \mathfrak{g}}$. 
Now, to solve the equation $f(\ts,\chi_\ts)=0$ by using the implicit function theorem, all that is necessary is to design 
$\chi_0 \in \mathrm{Im} \, (\mathrm{ad}_{M_0})_{|\mathfrak{g} }$ so that (i)
$$
f(0,\chi_0) =  \mathrm{ad}_{M_0}   \left( R_0- \mathrm{ad}_{M_0} \chi_0 \right) = 0 
$$
and (ii) $\mathrm{d}_\chi f(0,\chi_0) = - \mathrm{ad}_{M_0}^2 :  \mathrm{Im} \, (\mathrm{ad}_{M_0})_{| \mathfrak{g}} \to  \mathrm{Im} \, (\mathrm{ad}_{M_0})_{| \mathfrak{g}}$ is invertible. First we note that (i) is a direct consequence of (ii). To prove (ii), it is enough to apply \cite[Lemma 3.6]{bernier23scs} and the 
rank-nullity theorem. Indeed, this lemma states that  $M_0$ is diagonalizable on $\mathbb{C}$ if and only if the kernel and the image of $\mathrm{ad}_{M_0}$ are supplementary.
\end{proof}

\subsection{Linear unitary flows}

While in subsection \ref{s_spectral} we have seen that an alternating-conjugate splitting method $S_\ts$ with skew-Hermitian matrices
$A_j$ is consistent with a unitary flow $\e^{i\ts H}$ (with $H$ Hermitian) and that its spectrum satisfies $\sigma(S_\ts) = \overline{\mathrm{inv}(\sigma(S_\ts))}$,
the following theorem implies that, typically\footnote{{In fact, this property is valid for almost all Hermitian matrices $H$: the eigenvalues of almost all Hermitian matrices are simple, and this set of matrices is the complementary set of the algebraic submanifold of the matrices whose discriminant is equal to zero.}}, $S_h$ is indeed similar to a unitary matrix.

\begin{theorem}
\label{thm:unitary}
Let  $H \in \mathbb{C}^{N\times N}$ be a complex matrix and let $S_\ts \in \mathbb{C}^{N\times N}$ be a family of invertible complex matrices depending smoothly on $\ts\in \mathbb{R}$ such that
\begin{itemize}
\item the spectrum of $S_h$ satisfies $\sigma(S_\ts) =\overline{\mathrm{inv}(\sigma(S_\ts))}$
\item $S_\ts$ is consistent with $\exp(i\ts H)$, i.e. there exists $p\geq 1$ such that 
$
\displaystyle S_\ts \mathop{=}_{\ts\to 0} \e^{i\ts H} + \mathcal{O}(\ts^{p+1}),
$
\item the eigenvalues of $H$ are real and simple.
\end{itemize}
Then there exist 
\begin{itemize}
\item $D_\ts$,  a family of real diagonal matrices depending smoothly on $\ts$
\item $P_\ts$, a family of complex invertible matrices depending smoothly on $\ts$,
\end{itemize}
such that $P_\ts = P_0 + \mathcal{O}(\ts^p)$, $D_\ts = D_0 + \mathcal{O}(\ts^p)$  and, provided that $|\ts|$ is small enough, 
\begin{equation}
\label{eq:what_we_want}
S_\ts = P_\ts \, \e^{i\ts D_\ts} \, P_\ts^{-1}.
\end{equation}
\end{theorem}
Using this theorem we can deduce as a corollary that the actions of $\exp(i\ts H)$ are almost preserved
 for all times as well as the associated norm and energy whenever $H$ is Hermitian. Specifically, 
\begin{corollary} 
\label{cor:simple} 
In the setting of Theorem \ref{thm:unitary}, there exists a constant $C>0$ such that, provided that $|\ts|$ is small enough,  for all $u\in \mathbb{C}^N$ and all 
eigenvalues $\omega \in \sigma(H)$, one has
\begin{equation} \label{eq:cor_est_eigs}
 \sup_{n\geq 0} \, \Big||\Pi_\omega S_\ts^{n} u| - |\Pi_\omega  u| \Big| \leq C |\ts|^p |u|,
\end{equation}
where $\Pi_\omega$ denotes the spectral projector onto $\mathrm{Ker}(H-\omega I_N)$. Moreover, if $H$ is Hermitian, the norm and the energy are almost conserved, 
in the sense that, for all $u\in \mathbb{C}^N$, it holds that 
\begin{equation} \label{eq:cor_mass_energ}
  \sup_{n \in \mathbb{Z}} \,  \big| \mathcal{M}(S_\ts^n u) - \mathcal{M}( u)  \big|\leq C |\ts|^p |u|^2  \qquad \mbox{ and } \qquad 
  \sup_{n \in \mathbb{Z}} \, \big| \mathcal{H}(S_\ts^n u) - \mathcal{H}( u)  \big| \leq C |\ts|^p |u|^2, 
\end{equation}
where $\mathcal{M}( u) = |u|^2 $ and $\mathcal{H}( u) = \overline{u}^T H u $.
\end{corollary}
{The proof of this result is indeed identical to  Corollary 3.2. of \cite{bernier23scs}, so there is no need to repeat it here. The main
arguments are that (i) $S_h^n = P_0 \e^{i n h D_h} P_0^{-1} + \mathcal{O}(h^p)$, where the the implicit constant in the $\mathcal{O}$ terms
does not depend on $n$  and (ii) the explicit form for the spectral projectors associated with $H$ (due to the fact that $H$ has
 simple eigenvalues).}

\

\begin{proof}[Proof of Theorem \ref{thm:unitary}]
\noindent $\bullet$ \emph{Step 1 : Backward error analysis and spectrum.}
 Provided that $|\ts|$ is small enough, we define (in virtue of holomorphic functional calculus)
\begin{equation} \label{Hh}
H_{\ts} := (i \ts)^{-1} \log S_\ts .
\end{equation}
It follows that
$$
S_{\ts} = \e^{i\ts H_{\ts} }.
$$
By the consistency assumption, we deduce that $H_{\ts}$ depends smoothly on $\ts$ and that
\begin{equation}
\label{eq:consis}
 H_\ts = H+ \mathcal{O}(\ts^{p}). 
\end{equation}
Then, thanks to the usual properties of the homolomorphic functional calculus, we have (for $\ts \neq 0$)
$$
\sigma (H_{\ts}) = (i \ts)^{-1} \log \sigma(S_\ts).
$$
Since, by assumption, $\sigma(S_\ts) =\overline{\mathrm{inv}(\sigma(S_\ts))}$, it follows that the spectrum of $\sigma (H_{\ts})$ is invariant by complex conjugation:
\begin{equation}
\label{eq:spectrum_conj}
\overline{\sigma (H_{\ts})} = -  (i \ts)^{-1} \overline{\log \sigma(S_\ts)} = (i \ts)^{-1} \log(\overline{\mathrm{inv}(\sigma(S_\ts))}) = (i \ts)^{-1} \log \sigma(S_\ts) = \sigma (H_{\ts}).
\end{equation}

\noindent $\bullet$ \emph{Step 2 : Normal form.}
Since by hypothesis the eigenvalues of $H$ are real and simple, $H$ is diagonalizable and so,
without loss of generality, we assume that it is diagonal, that is, 
\[
H = \begin{pmatrix} \omega_1 \\ & \ddots \\ & & \omega_N \end{pmatrix}, 
\quad \mbox{ where} \  \omega_1 <\cdots < \omega_N \mbox{ denote the eigenvalues of $H$}.
\]
Taking into account the consistency estimate \eqref{eq:consis}, we apply Lemma \ref{lem:diag_cont} to $H$, so that 
we get a family $\chi_\ts$ of complex matrices depending smoothly on $\ts$, such that for sufficiently small $|\ts|$,
$$
[P_\ts H_\ts  P_{\ts}^{-1} , H ] =0 \qquad \mbox{ where } \qquad P_\ts := \e^{\ts^p\chi_\ts}.
$$
Since $H$ is diagonal and its eigenvalues are simple, it just means that $P_\ts H_\ts  P_{\ts}^{-1}$ is also diagonal. Therefore, we set
$$
D_\ts :=   \begin{pmatrix} \omega_{1,h} \\ & \ddots \\ & & \omega_{N,h} \end{pmatrix} := P_\ts H_\ts  P_{\ts}^{-1}.
$$ 
By construction, the complex numbers $\omega_{j,h}$ depend smoothly on $h$ and by consistency we have
$$
 \omega_{j,h} = \omega_j + \mathcal{O}(h^p).
$$
To conclude the proof, we just have to prove that the numbers $ \omega_{j,h}$ are real.

First, we note that $ \omega_{j,h} $ are the eigenvalues of $H_{\ts}$. Therefore, since the spectrum of $H_\ts$ is invariant by complex conjugation (see \eqref{eq:spectrum_conj}), it is true that
$$
\{ \omega_{j,h} \ | \ 1\leq j\leq N\} = \{ \overline{\omega_{j,h}} \ | \ 1\leq j\leq N\},
$$
but since the eigenvalues of $H$ are real (by assumption), then
$$
 \overline{\omega_{j,h}} = \omega_j + \mathcal{O}(h^p).
$$
If follows that if $h$ is small enough to ensure that
$$
|h|^p \lesssim \min_{k\neq \ell} |\omega_k - \omega_\ell|
$$
then for all $j\in \llbracket 1,N\rrbracket$, we have
$$
\omega_{j,h}  = \overline{\omega_{j,h}},
$$
i.e. they are real numbers.
\end{proof}

Theorem \ref{thm:unitary} thus explains the good behavior of alternating-conjugate splitting methods displayed in Figures \ref{figure.1} and \ref{figure.2},
when the eigenvalues of $H$ are real and simple: for sufficiently small $|h|$, they are conjugate to unitary maps. However, it does not account for the same behavior
observed in practice when $H$ has repeated eigenvalues (Figure \ref{figure.3} right). This problem is dealt with in the Appendix and is related with the
structure of the operator $H_h$ introduced in \eqref{Hh}.

\subsection{Hermitian flows}

In the case of an alternating-conjugate splitting method $S_h$ approximating the flow generated by a Hermitian matrix, the following theorem allows one to
conclude that typically\footnote{{Here again, the property is valid for almost all Hermitian matrices $H$, since the eigenvalues of almost all Hermitian matrices
are simple and this set of matrices is the complementary set of an algebraic submanifold.}}, $S_\ts$ is similar to a real diagonal matrix (that is, in particular, a Hermitian matrix). We recall that the necessary assumption on the spectrum, $\sigma(S_\ts) = \overline{\sigma(S_\ts)}$, is always satisfied if the matrices 
$A_j$ are real symmetric, whereas if they are Hermitian, $\Psi_h$ has to be symmetric-conjugate (see \eqref{eq:sym_conj}).

\begin{theorem}
\label{thm:hermitian}
Let  $H \in \mathbb{C}^{N\times N}$ be a complex matrix and let $S_\ts \in \mathbb{C}^{N\times N}$ be a family of complex matrices depending smoothly on $\ts\in \mathbb{R}$ such that
\begin{itemize}
\item the spectrum of $S_h$ satisfies $\sigma(S_\ts) = \overline{\sigma(S_\ts)}$,
\item $S_\ts$ is consistent with $\exp(\ts H)$, i.e. there exists $p\geq 1$ such that
$
\displaystyle S_\ts \mathop{=}_{\ts\to 0} \e^{\ts H} + \mathcal{O}(\ts^{p+1}),
$
\item the eigenvalues of $H$ are real and simple.
\end{itemize}
Then there exist 
\begin{itemize}
\item $D_\ts$,  a family of real diagonal matrices depending smoothly on $\ts$,
\item $P_\ts$, a family of complex invertible matrices depending smoothly on $\ts$,
\end{itemize}
such that $P_\ts = P_0 + \mathcal{O}(\ts^p)$, $D_\ts = D_0 + \mathcal{O}(\ts^p)$  and, provided that $|\ts|$ is small enough, 
\begin{equation}
\label{eq:what_we_want_herm}
S_\ts = P_\ts \, \e^{\ts D_\ts} \, P_\ts^{-1}.
\end{equation}
\end{theorem}
\begin{proof} Up to some multiplicative factors $i$, the proof is exactly the same as the one of Theorem \ref{thm:unitary}, and so we omit it.
\end{proof}

\section{Linear Hamiltonian systems} 
\label{sec.4}

Let us analyze now how alternating-conjugate methods approximate the dynamics of the important class of linear Hamiltonian systems.
In canonical real variables, a real (resp. complex) linear Hamiltonian system is given by a linear evolution equation of the form
\begin{equation} \label{Ham.1}
 \dot{y} = J H y,
\end{equation}
where $y\in \mathbb{R}^{2N}$, $J \in \mathbb{R}^{2N \times 2N}$ is the matrix of the canonical symplectic form on $ \mathbb{R}^{2N}$,
$$
J = \begin{pmatrix}  O_N & I_N \\ -I_N & O_N
\end{pmatrix}
$$
and $H$ is a real (resp. complex) symmetric $2N\times 2N$ matrix. As is well known, the corresponding Hamiltonian function $\mathcal{H}(y)$
is a constant of the motion, i.e., 
\begin{equation}
\label{eq:def_hamil}
\forall \, y\in \mathbb{R}^{2N}, \ \ \forall \, t\in \mathbb{R}, \qquad \mathcal{H}( \e^{tJH} y) = \mathcal{H}(y), \quad \mbox{ where } \quad 
\mathcal{H}(y) = \frac{1}{2} \, y^T H y,
\end{equation}
and the flow is symplectic, 
$$
(\e^{tJH})^T J \, \e^{tJH} = J, \quad \forall \, t\in \mathbb{R}.
$$
Notice that these two properties can be extended by analyticity to all $t\in \mathbb{C}$, and that if $P$ is a symplectic change of variables, then
\begin{equation}
\label{eq:sympl_inv}
P^{-1} J H P = J P^T H P.
\end{equation}
The set of all real (resp. complex) symplectic matrices form a matrix Lie group, the real or complex symplectic group $\mathrm{Sp}_{2N}$. Moreover, it can be shown that, in a neighborhood of the identity, any linear real (resp. complex) symplectic map is given by the flow of a real (resp. complex) linear
Hamiltonian system. 

It turns out that splitting methods are particularly well suited to discretize such systems. 
More specifically, let us consider a classical splitting method applied to a real linear Hamiltonian system, i.e.
$$
S_h = \e^{\beta_1 h J H_1 } \, \cdots \, \e^{\beta_k h J H_k }  = \e^{h J H} +\mathcal{O}(h^2),
$$
where $\beta_1,\cdots, \beta_k \in \mathbb{R}$ are real numbers and $H,H_1,\cdots,H_k$ are real symmetric matrices. By using backward error analysis 
techniques, it can be shown that there exits a family of real symmetric matrices $H_h$ depending smoothly on $h$ such that if $|h|$ is sufficiently small, one has
\begin{equation}
\label{eq:consist_ham}
S_h = \e^{hJH_h}, \qquad \mbox{ with } \qquad H_h = H + \mathcal{O}(h).
\end{equation}
It follows that $H_h$ is a constant of the motion of the numerical flow. Furthermore, if $H$ is positive definite then $H_h$ is also so (for $|h|$ small enough), the numerical flow is globally bounded and the Hamiltonian $\mathcal{H}$ (defined in \eqref{eq:def_hamil}) is almost globally preserved for all times, i.e.
$$
\forall \, y\in \mathbb{R}^{2N}, \quad \sup_{n\in \mathbb{Z}}| S_h^n y | \lesssim |y| \qquad \mbox{ and } \qquad \sup_{n\in \mathbb{Z}} \big|\mathcal{H}(S_h^n y) - \mathcal{H}(y)    \big| \lesssim |h| |y|^2.
$$

\medskip

If instead of a classical splitting method (i.e., with real coefficients), we consider a splitting method with complex coefficients, $\beta_1,\cdots, \beta_k \in \mathbb{C}$, then the numerical flow is still symplectic, so that it can be seen as a Hamiltonian flow which is consistent with the original one (i.e.  \eqref{eq:consist_ham} holds) but now the modified $H_h$ is a complex matrix. As a consequence, if $H$ is positive definite, the preservation of $H_h$ by the numerical flow does not ensure that it remains bounded.

At this point it is worth describing more precisely the dynamics of the class of linear Hamiltonian systems. As is well known (see e.g. \cite{hormander95sco}), 
if $H$ is definite positive then, up to a symplectic change of variable, the corresponding Hamiltonian flow is a symplectic rotation or, equivalently, the associated Hamiltonian system is completely integrable: there exists a real symplectic matrix $P$ and some frequencies $\omega_1,\ldots,\omega_N >0$ such that
\begin{equation}
\label{eq:conj_int}
 H  = P^T D_\omega \, P \quad \mbox{ where } \quad D_\omega:=\begin{pmatrix} \Delta_\omega \\ & \Delta_\omega \end{pmatrix} \quad \mbox{ and } \quad \Delta_\omega := \begin{pmatrix} \omega_1 \\ & \ddots \\ & & \omega_N  \end{pmatrix}
\end{equation}
and thus
$$
\e^{t JH} = \e^{t J  P^T D_\omega P } \mathop{=}^{\eqref{eq:sympl_inv}} P^{-1} \e^{t J   D_\omega  }P   =P^{-1}\begin{pmatrix} \cos(t \Delta_\omega) & \sin(t \Delta_\omega) \\ - \sin(t \Delta_\omega) & \cos(t \Delta_\omega) \end{pmatrix} P.
$$
More generally, if a Hamiltonian flow is similar to a symplectic rotation (the frequencies $\omega_j$ do not necessarily have to be positive) then it is globally bounded even if the corresponding matrix $H$ is not positive definite. In fact, linear Hamiltonian systems that are completely integrable have a great deal of
 constant of the motions, called actions, which control all the dynamics. Specifically, for all $ t \in \mathbb{R}$, all $y = (q,p) \in \mathbb{R}^{2N}$ and all $ j\in \llbracket 1,N \rrbracket$,
\begin{equation}
\label{eq:def_actions}
 \mathcal{I}_j (\e^{tJH} y) = \mathcal{I}_j (y),  \qquad \mbox{ where } \quad \mathcal{I}_j(y) := \mathcal{J}_j(Py), \quad \mbox{ and } \quad 
 \mathcal{J}_j(p,q) := \frac{1}{2}(p_j^2+q_j^2) .
\end{equation}

\medskip

What happens when an alternating-conjugate splitting method $S_h$ is applied to a real linear Hamiltonian system? Before addressing this question from a theoretical
point of view, we carry out a numerical experiment with a system of type \eqref{eq:conj_int}. More specifically, we take $N=3$, generate a real symmetric
matrix $H$ with three different frequencies $\omega_1, \omega_2, \omega_3 > 0$, split it into two parts $H= A + B$, 
and check how the imaginary part of the numerical solution obtained with
schemes \eqref{scheme3}, \eqref{schemes4} and \eqref{ac.1} evolves with time. In doing so, we have to make the obvious replacements of $i A$ and $i B$ 
by $J A$ and $J B$, respectively, in the previous methods. The results are displayed in Figure \ref{figure.4}.

\begin{figure}[!ht] 
\centering
  \includegraphics[width=.49\textwidth]{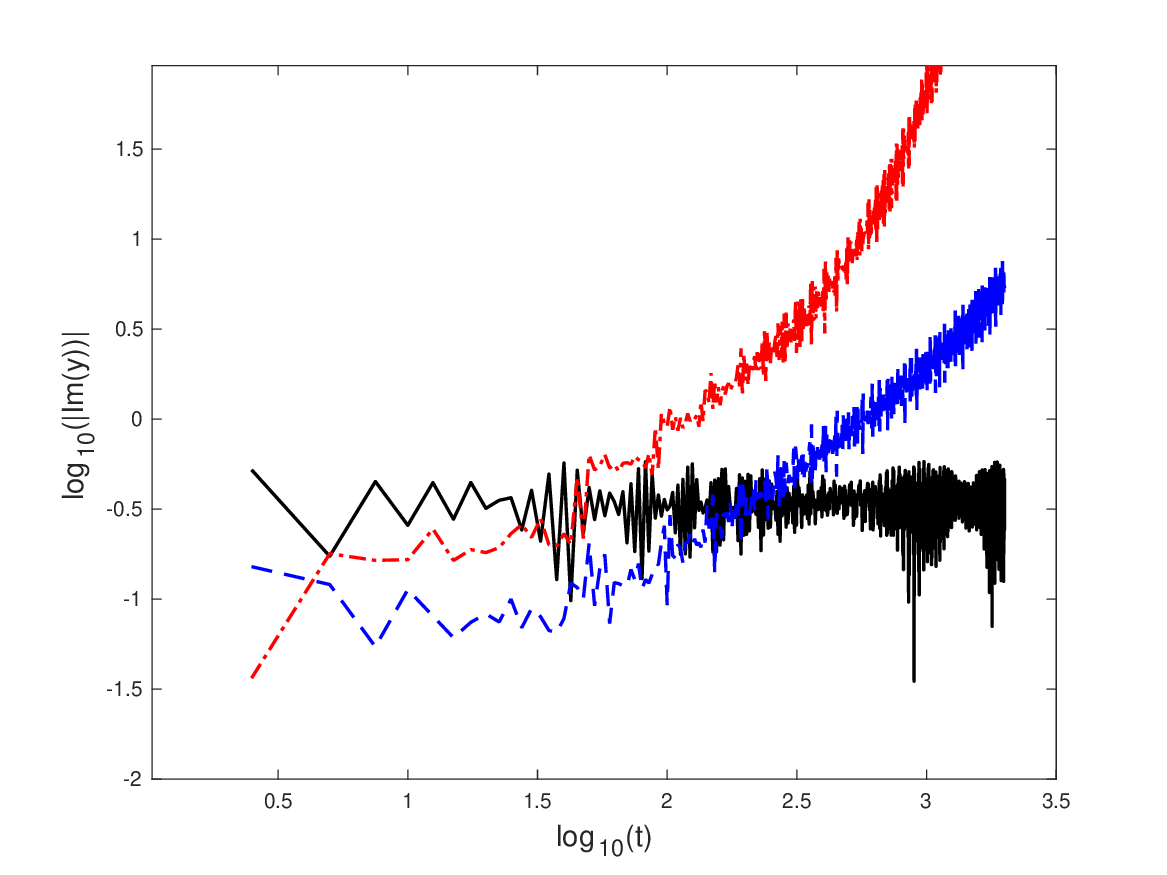}
  \includegraphics[width=.49\textwidth]{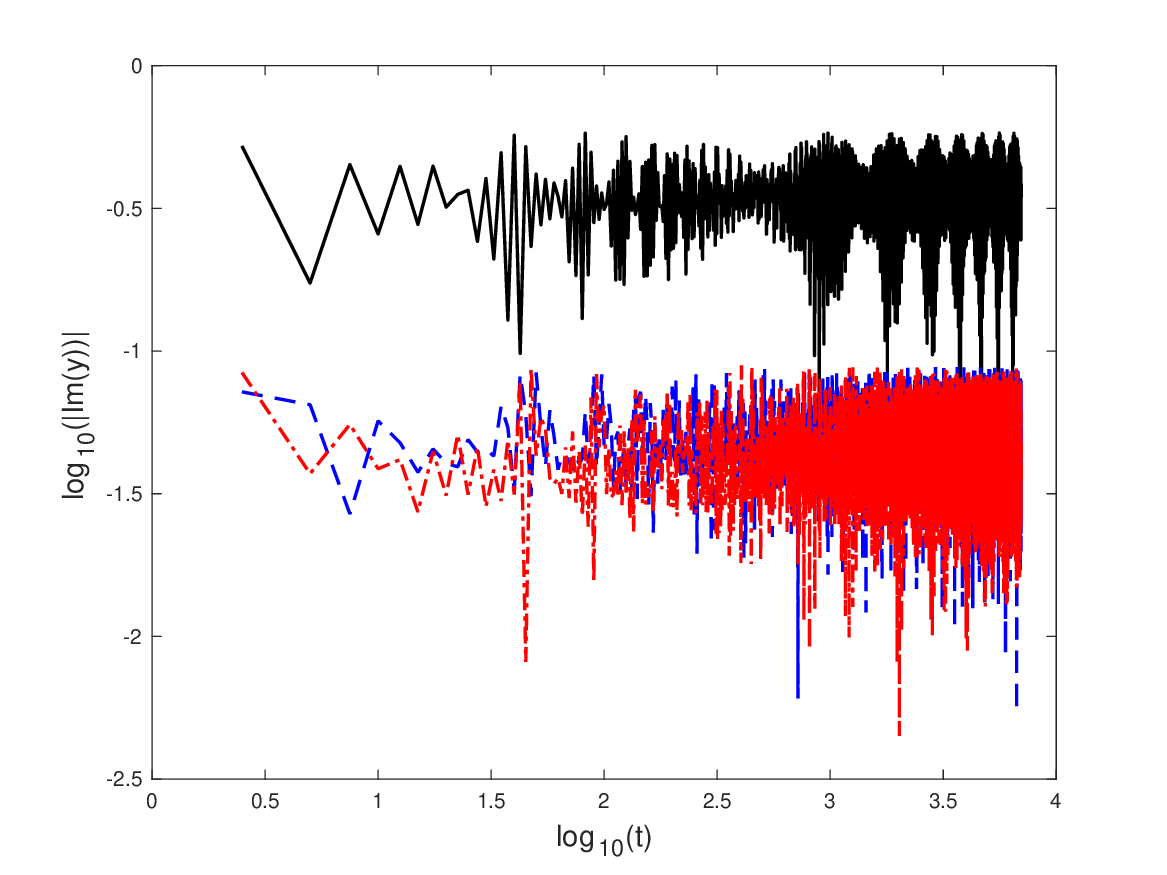}
\caption{\label{figure.4} \small Imaginary part of the numerical solution as a function of time for a linear Hamiltonian system with $N=3$ and step size $h=2.5$. 
Left panel: schemes
$\Psi_{h}^{[3]}$ (black {solid}), $\Psi_{h,sc}^{[4]}$ (blue {dashed}) and
$\Psi_{h,p}^{[4]}$ (red {dash-dotted}). Right panel: results achieved by the alternating-conjugate version \eqref{ac.1} on the same problem. Only the numerical solution
obtained by alternating-conjugate methods remains almost real for very long times.}
\end{figure}

We clearly see that both palindromic and symmetric-conjugate methods introduce an imaginary part into the numerical solution that increases with time, whereas for
the corresponding alternating-conjugate this imaginary part remains small for very long times, and this takes place even for not so small values of the step size $h$.

\begin{figure}[!ht] 
\centering
  \includegraphics[width=.49\textwidth]{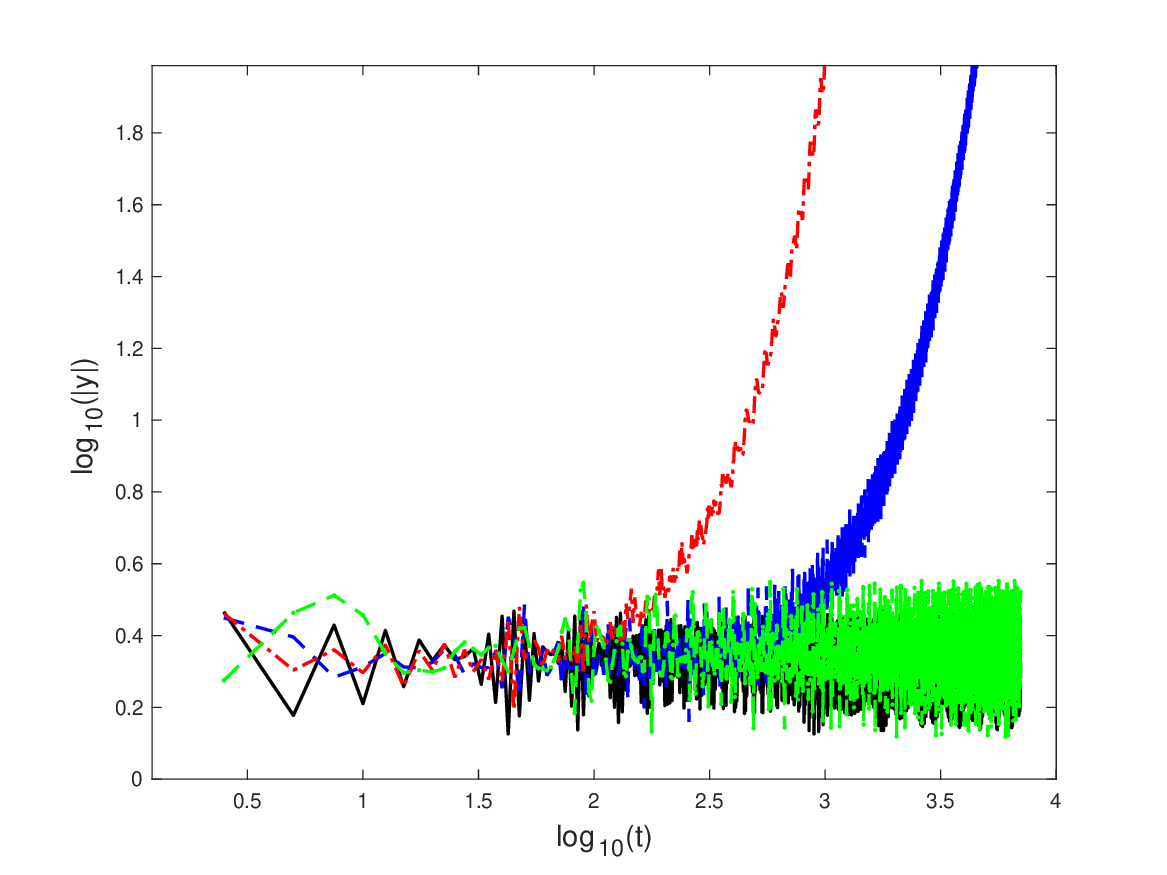}
  \includegraphics[width=.49\textwidth]{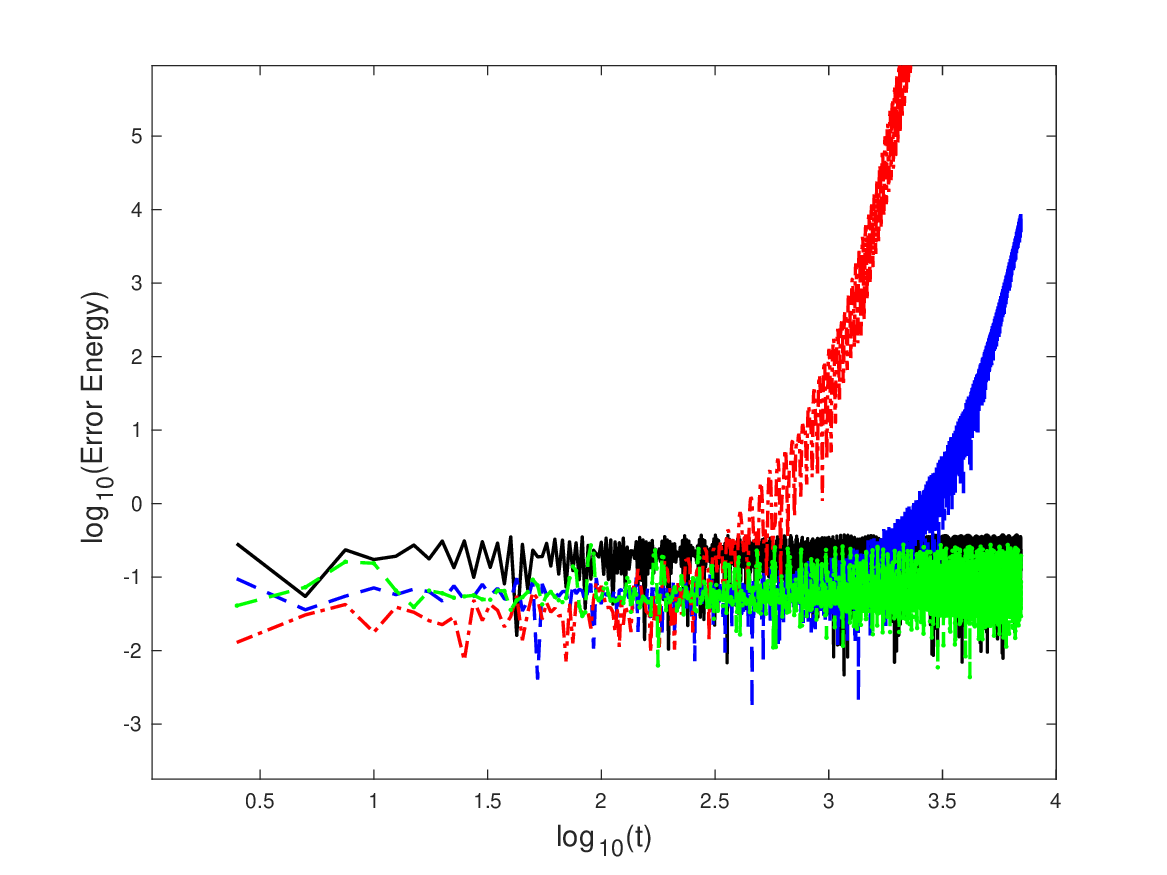}
\caption{\label{figure.5} \small Norm of the numerical solution (left) and error in energy (right) as a function of time for a linear Hamiltonian system with $N=3$ 
and step size $h=2.5$ for the following schemes: 
$\Psi_{h}^{[3]}$ (black {solid}), $\Psi_{h,sc}^{[4]}$ (blue {dashed}), 
$\Psi_{h,p}^{[4]}$ (red {dash-dotted}) and the new 4th-order alternating-conjugate method \eqref{nac4} (green {dashed}).  Only alternating-conjugate schemes provide a correct
description of the dynamics.}
\end{figure}

We can also check how the norm $|y|$ of the numerical solution of \eqref{Ham.1} and the error in the energy $\mathcal{H}(y)$ obtained by the different 
schemes evolve with time. The corresponding results are depicted in Figure \ref{figure.5} for methods $\Psi_{h}^{[3]}$ (black {solid}), $\Psi_{h,sc}^{[4]}$ (blue {dashed}),
$\Psi_{h,p}^{[4]}$ (red {dash-dotted}) and a new alternating-conjugate method of order 4 \eqref{nac4} 
presented in section \ref{sec.5} (green) which is \emph{not} obtained by concatenating
palindromic or symmetric-conjugate schemes. Notice that only alternating-conjugate methods provide a correct qualitative description of the dynamics.
The same behavior is observed if one considers instead the actions \eqref{eq:def_actions}.

As in the case of the unitary and Hermitian flows analyzed before, the spectral properties of alternating-conjugate maps play a fundamental role in explaining these
results. Specifically, if the system is completely
integrable, then the fact that $\sigma(S_\ts) = \overline{\sigma(S_\ts)}$ guarantees that the numerical flow is conjugated to a symplectic rotation 
and so it is globally bounded. The precise statement is the following:

\begin{theorem}
\label{thm:simple_sympl}
Let  $H \in \mathbb{R}^{2N\times 2N}$ be a real symmetric matrix and let $S_\ts \in \mathbb{C}^{2N\times 2N}$ be a family of complex symplectic matrices depending smoothly on $\ts\in \mathbb{R}$ such that
\begin{itemize}
\item the spectrum of $S_h$ satisfies $\sigma(S_\ts) = \overline{\sigma(S_\ts)}$,
\item $S_\ts$ is consistent with $\exp(\ts J H)$, i.e. there exists $p\geq 1$ such that
$
\displaystyle S_\ts \mathop{=}_{\ts\to 0} \e^{\ts J H} + \mathcal{O}(\ts^{p+1}),
$
\item the Hamiltonian system associated with $H$ is completely integrable, i.e. there exists a real symplectic linear matrix $P  \in \mathbb{R}^{2N \times 2N}$ and some real frequencies $\omega_1,\cdots,\omega_N$ such that $H  = P^T D_\omega P$ (see \eqref{eq:conj_int}),
\item the frequencies $\omega$ are {\sl non resonant}, in the sense that 
the numbers $\omega_1, \ldots, \omega_N, -\omega_1, \ldots, -\omega_N$ are all distinct.
\end{itemize}
Then there exist 
\begin{itemize}
\item $\omega_\ts$,  a vector of real frequencies depending smoothly on $\ts$,
\item $P_\ts$, a family of complex symplectic matrices depending smoothly on $\ts$,
\end{itemize}
such that $P_\ts = P + \mathcal{O}(\ts^p)$, $\omega_\ts = \omega + \mathcal{O}(\ts^p)$  and, provided that $|\ts|$ is small enough, 
\begin{equation}
\label{eq:what_we_want_symp}
S_\ts = P_\ts^{-1} \, \e^{\ts J D_{\omega_\ts}} \, P_\ts.
\end{equation}
\end{theorem}
As a corollary, we can show that the dynamics is bounded, the numerical solution remains almost real for all times and the actions and the Hamiltonian
function of the original system are almost preserved. Specifically,
\begin{corollary} 
\label{cor:nice} 
In the setting of Theorem \ref{thm:simple_sympl}, there exists a constant $C>0$ such that, provided that $|\ts|$ is small enough,  for all $y = (q,p) \in \mathbb{R}^{2N}$, 
\begin{itemize}
\item the numerical solution remains globally bounded in time and its imaginary part remains very small for all times
$$
\sup_{n\in \mathbb{Z}} |S_h^n y| \leq C |y| \qquad \mbox{ and } \qquad \sup_{n\in \mathbb{Z}} |\Im (S_h^n y)| \leq C |\ts|^p|y|,
$$
\item the actions and the Hamiltonian of the original system (defined respectively in  \eqref{eq:def_actions} and \eqref{eq:def_hamil} ) are almost preserved for all times
$$
\sup_{n\in \mathbb{Z}} \sup_{1\leq j\leq N} |\mathcal{I}_j (S_h^n y) - \mathcal{I}_j (y)| \leq C |\ts|^p |y|^2 \qquad \mbox{ and } \qquad \sup_{n\in \mathbb{Z}}  |\mathcal{H} (S_h^n y) - \mathcal{H} (y)| \leq C|\ts|^p  |y|^2.
$$
\end{itemize}
\end{corollary}

We first proceed to prove the corollary.
\begin{proof}[Proof of Corollary \ref{cor:nice}] Since $S_\ts$ is conjugated to a symplectic rotation, it is clear that the numerical solution remains bounded. Concerning the imaginary part, since $y$ is a real vector, then
\begin{equation*}
\begin{split}
|\Im S_h^n y| = |(\Im S_h^n) y| \leq |\Im S_h^n| |y| &= |\Im  ( P_\ts^{-1} \, \e^{n\ts J D_{\omega_\ts}} \, P_\ts )| |y| \\ &=   | \Im(P_\ts^{-1} )\, \e^{n\ts J D_{\omega_\ts}} \, \Re(P_\ts) +   \Re(P_\ts^{-1}) \, \e^{n\ts J D_{\omega_\ts}} \, \Im(P_\ts) | |y|  \\ 
&\lesssim (|\Im(P_\ts^{-1}) | + |\Im(P_\ts)|) |y| \lesssim |\ts|^p |y|,
\end{split}
\end{equation*}
the last estimate coming from the fact that $P_\ts = P+\mathcal{O}(h^p)$ and that $P$ is a real symplectic matrix.

With respect to the actions, we have
$$
\mathcal{H} = \sum_{j=1}^N \omega_j \mathcal{I}_j,
$$
so that to get the almost preservation of the Hamiltonian, it is enough to prove the almost preservation of the actions. But since $S_\ts$ is similar to a symplectic
rotation, there exists modified actions which are preserved by the numerical flow, that is for all $n\in \mathbb{Z}$, all $\ts$ small enough and all $y\in \mathbb{R}^{2N}$,
$$
 \mathcal{I}_{j,h} (S_h^n y) =  \mathcal{I}_{j,h} (y) \qquad \mbox{ where } \qquad \mathcal{I}_{j,h}(y) := \mathcal{J}_j(P_\ts y)
$$
where $\mathcal{J}_j$ is defined in \eqref{eq:def_actions}. Recalling that $ \mathcal{I}_{j}(y) := \mathcal{J}_j(P y)$, the almost global preservation of the action is then just a consequence of the estimate $P_\ts = P + \mathcal{O}(h^p)$.
\end{proof}

Finally, we prove Theorem \ref{thm:simple_sympl}.

\begin{proof}[Proof of Theorem \ref{thm:simple_sympl}]  Since the matrices $S_\ts$ are complex symplectic, there exists a family of complex symmetric matrices $H_\ts$, depending smoothly on $\ts$, such that, if $|\ts|$ is small enough, then
$$
S_{\ts} = \e^{\ts J H_{\ts} }.
$$
Due to the consistency assumption, it holds that
\begin{equation}
\label{eq:cons_Hh}
 H_\ts \mathop{=}_{h\to 0} H+ \mathcal{O}(\ts^{p}). 
\end{equation}
Also by assumption, $H  = P^T D_\omega P$, so that, since $P$ is symplectic, we have (eq. \eqref{eq:sympl_inv})
$$
JH = P^{-1} JD_\omega P.
$$
On the other hand, 
\begin{equation}
\label{eq:useful}
L JD_\omega L^{-1}=  i \begin{pmatrix} \Delta_\omega \\ & -\Delta_\omega \end{pmatrix}, \quad \mbox{ where } \quad L: = \frac12 \begin{pmatrix} (1-i)I_N & -(1+i) I_N \\  (1-i) I_N & (1+i) I_N  \end{pmatrix}
\end{equation}
and thus the eigenvalues of $JH$ are $i\omega_1, \ldots, i\omega_N, -i\omega_1, \ldots, -i\omega_N$.

\medskip

Our goal now is to apply the technical Lemma \ref{lem:diag_cont} in $\mathfrak{g} = \mathfrak{sp}_{2N}(\mathbb{C})$, the Lie algebra of 
complex symplectic matrices, that is
$$
\mathfrak{sp}_{2N}(\mathbb{C}) = \{ J Q \ | \ Q \in \mathbb{C}^{2N \times 2N} \quad \mathrm{and} \quad Q^T = Q\}
$$
and to the family 
$$
M_\ts := L P(J H_h)P^{-1} L^{-1}.
$$
To do that, we have first to check that $M_\ts \in \mathfrak{sp}_{2N}(\mathbb{C})$. But this is clear, since $H_h$ is symmetric and $P$, $L$ are symplectic:
$$
M_\ts = J (P^{-1} L^{-1})^T H_h (P^{-1} L^{-1})  \in \mathfrak{sp}_{2N}(\mathbb{C}).
$$
and, moreover, $M_0$ is diagonal,
$$
M_0 = i \begin{pmatrix} \Delta_\omega \\ & -\Delta_\omega \end{pmatrix}.
$$
As a consequence of Lemma \ref{lem:diag_cont}, we get a family of complex symmetric matrices $\chi_\ts$, depending smoothly on $\ts$, such that, if $|\ts|$ is small enough, then
$$
[\widetilde{M_\ts}  , M_0 ] =0, \qquad \mbox{ where } \qquad \widetilde{M_\ts} := \e^{\ts^pJ \chi_\ts} M_\ts \, \e^{-\ts^pJ \chi_\ts}.
$$
Since $M_0$ is diagonal and by assumption its eigenvalues are simple, then $\widetilde{M_\ts}$ is also diagonal, i.e. there exist two vectors of complex 
numbers $\omega_\ts,\lambda_\ts \in \mathbb{C}^{N}$, depending smoothly on $\ts$, such that, provided that $|\ts|$ is small enough, 
$$
\widetilde{M_\ts} = i \begin{pmatrix} \Delta_{\omega_\ts} \\ & -\Delta_{\lambda_\ts} \end{pmatrix}.
$$

\medskip

Since by construction $\chi_\ts$ is symmetric, it follows that $\e^{\ts^pJ \chi_\ts} $ is symplectic and that $J\widetilde{M_\ts}$ is itself
a symmetric matrix. Therefore,
$$
\omega_\ts = \lambda_\ts
$$
and so, provided that $|\ts|$ is small enough,
$$
 JD_{\omega_\ts} = L^{-1}\widetilde{M_\ts}L =  L^{-1} \e^{\ts^pJ \chi_\ts}  L P(J H_h)P^{-1} L^{-1} \e^{-\ts^pJ \chi_\ts} L.
$$
As a consequence, setting
\begin{equation}
\label{eq:defPts}
P_\ts :=L^{-1} \e^{\ts^pJ \chi_\ts}  L P,
\end{equation}
we have proven, as expected, that if $|\ts|$ is small enough, 
\begin{equation}
\label{eq:partial_conclusion}
S_\ts = \e^{\ts J H_\ts} = P_\ts^{-1} \, \e^{\ts J D_{\omega_\ts}} \, P_\ts.
\end{equation}
It only remains to prove the consistency estimates and that the frequencies $\omega_\ts$ are real.

\medskip

For the consistency, by definition of $P_\ts$ (see \eqref{eq:defPts}), it is clear that
$
P_\ts = P + \mathcal{O}(h^p).
$
Then, for the frequencies $\omega$, we have proven that
$$
\widetilde{M_\ts}  =  \e^{\ts^pJ \chi_\ts} L P J (H_\ts-H) P^{-1} L^{-1} \e^{-\ts^pJ \chi_\ts} +  \e^{\ts^pJ \chi_\ts} M_0 \, \e^{-\ts^pJ \chi_\ts} ,
$$
which ensures, by definition of $\omega_\ts$ and by consistency of $H_\ts$  (see \eqref{eq:cons_Hh}), that
$$
\omega_\ts = \omega + \mathcal{O}(h^p).
$$

\medskip

Finally, we just have to prove that the frequencies $\omega_\ts$ are real. By relations \eqref{eq:partial_conclusion} and \eqref{eq:useful}, we know that 
$$
\sigma(J H_h) = i \{ \omega_{h,1}, \ldots, \omega_{h,N},  -\omega_{h,1}, \ldots, -\omega_{h,N} \}.
$$
On the other hand, by assumption, the spectrum of $S_h$ is invariant by complex conjugation, and so,  if $|\ts|$ is small enough,
$$
\sigma(J H_h) = \overline{\sigma(J H_h)}.
$$
which implies that
$$
\{ \omega_{h,1}, \ldots, \omega_{h,N},  -\omega_{h,1}, \ldots, -\omega_{h,N} \} = \{ \overline{\omega_{h,1}}, \ldots, \overline{\omega_{h,N}},  -\overline{\omega_{h,1}}, \ldots, -\overline{\omega_{h,N}} \} .
$$
Since the frequencies $\omega_1, \ldots, \omega_N$ are real, by consistency, we have
$$
\overline{\omega} = \omega + \mathcal{O}(h^p).
$$
As a consequence, under the assumption that $\omega_1, \ldots, \omega_N, -\omega_1, \ldots, -\omega_N$ are all distinct, and provided that
$$
|\ts|^p \lesssim \min\Big( \min_{j\neq \ell} \big| |\omega_j| - |\omega_\ell|\big|, \min_{1\leq j\leq N} |\omega_j| \Big)
$$
we deduce that for all $j\in \llbracket 1,N\rrbracket$, it holds that $\omega_{h,j} = \overline{\omega_{h,j} } $, i.e. that the frequencies are real.

\end{proof}

\section{Constructing alternating-conjugate schemes}
\label{sec.5}

We are interested now in constructing families of alternating-conjugate (AC) methods that generalize the schemes tested in previous sections. To proceed,
we consider the linear problem
\begin{equation} \label{prm1}
   \dot{u} = M u = (A_1 + A_2 + \cdots + A_m) u,
\end{equation}   
where $A_j \in \mathbb{C}^{N \times N}$, $j=1, \ldots, m \ge 2$, and apply composition \eqref{eq:psihtilde}. The Baker--Campbell--Hausdorff (BCH) formula allows us
to write the integrator $S_h$ as
\[
  S_h = \exp(M_h) = \exp( h M_1 + h^2 M_2 + h^3 M_3 + \cdots),
\]
where
\[
\begin{aligned}
  & M_1 = \sum_{j=1}^m (\alpha_j + \overline{\alpha}_j) A_j \\
  & M_2 = \sum_{j < k} (\alpha_j \alpha_k + \alpha_j \overline{\alpha}_k - \alpha_k \overline{\alpha}_j + \overline{\alpha}_j \overline{\alpha}_k) \, [A_j, A_k]
\end{aligned}
\]
The AC scheme $S_h$ is of order 2 if $M_1 = M$ and $M_2 = 0$. A straightforward computation shows that the only solutions of these equations are 
\begin{equation} \label{slt1}
   \alpha_1 = \cdots = \alpha_m \equiv a = \frac{1}{2} (1 \pm i ) \qquad \mbox{ and so } \qquad S_h = \Psi_{a h}^{[1]} \, \Psi_{\overline{a} h}^{[1]},
\end{equation}  
where $\Psi_h^{[1]}$ denotes the Lie--Trotter splitting. To achieve higher orders we might consider
the more general composition $S_h = \Psi_h \, \widetilde{\Psi}_h$ with
\begin{equation} \label{gen.split}
  \Psi_h = \prod_{j=1}^s \e^{h \alpha_1^{(j)} A_1} \, \e^{h \alpha_2^{(j)} A_2} \, \cdots \, \e^{h \alpha_m^{(j)} A_m},
\end{equation}
then apply sequentially the BCH formula to get the order conditions and finally to solve all of them. This procedure, however, is not feasible for $m > 2$, given the
extremely rapid growth of the number of order conditions required to achieve order $p \ge 3$. As is well known, this number is given by the dimension of the 
homogeneous subspace $\mathcal{L}_n$ of the free Lie algebra generated by $\{A_1, \ldots, A_m \}$ \cite{munthe-kaas99cia},
\[
  d(n,m) = \dim \mathcal{L}_n(A_1, \ldots, A_m) = \sum_{d|n} \mu(d) \, m^{n/d},
\]
where $\mu(d)$ is the M\"obius function. Thus, for instance, if $m=4$ then $d(n,4) = 4, 6, 20, 60, 204, 670$ for $n=1,\ldots, 6$. 
A more realistic approach consists then in taking $\Psi_h$ as a composition of the Lie--Trotter splitting with different parameters, or more generally, a composition of the form
\begin{equation} \label{psiac}
  \Psi_h = \Phi_{\alpha_1 h} \, \Phi_{\alpha_2 h} \, \cdots \, \Phi_{\alpha_r h},
\end{equation}
where $\Phi_h$ is any given consistent scheme and $\alpha_j \in \mathbb{C}$. This is the class of schemes we consider in the sequel.

\subsection{Modified vector field}
\label{sub.5.1}

We can get more insight the family of alternating-conjugate (AC) methods originated from \eqref{psiac} by analyzing the algebraic 
structure of the modified vector field associated with the scheme. 

If $\Phi_h$ is consistent (of first order, then there exists a series $Y(h) = \sum_{n \ge 1} h^n Y_n$ of matrices $Y_n$ satisfying
\cite{blanes24smf}
\[
  \Phi_h = \e^{Y(h)}, \qquad \mbox{ with } \qquad Y_1 = M,
\]
whereas if $\Phi_h$ is a 2nd-order time-symmetric method (for instance, the Strang splitting), then {$Y(h) = h M + \sum_{n \ge 1} h^{2n+1} Y_{2n+1}$}. In any case,
\[
  \Psi_h = \e^{Y(\alpha_1 h)} \, \e^{Y(\alpha_2 h)} \, \cdots \, \e^{Y(\alpha_r h)} 
\]
so that, by applying the BCH formula, we can write 
\[
  \Psi_h = \e^{K(h)},
\]
where $K(h)$ is the series
\begin{equation} \label{expreKh}
  K(h) = h \, k_{1,1} M + \sum_{\ell \ge 2} h^{\ell} \, \sum_{j=1}^{c(\ell)} k_{\ell,j} \, E_{\ell,j}.
\end{equation}
Here, $k_{i,j} \in \mathbb{C}$ depend on the coefficients $\alpha_n$ and $\{ E_{i,j} \}$ denote the elements of the homogeneous subspace
$\mathcal{L}_i(Y_1, Y_2, \ldots)$, of dimension $c(i)$. In particular, 
$E_{2,1} = Y_2$, $E_{3,1} = Y_3$, $E_{3,2} = [Y_1, Y_2]$, etc. {Thus, if $\Phi_h$ in \eqref{psiac} is of order 1, then
\begin{equation} \label{compo_o1}
\begin{aligned}
  & k_{1,1} = \sum_{i=1}^r \alpha_i, \qquad k_{2,1} = \sum_{i=1}^r \alpha_i^2, \qquad  k_{3,1} = \sum_{i=1}^r \alpha_i^3 \\
  & k_{3,2} = -\frac{1}{2} \sum_{i=1}^{r-1} \alpha_i^2 \sum_{j=i+1}^r \alpha_j + \frac{1}{2} \sum_{i=1}^{r-1} \alpha_i \sum_{j=i+1}^r \alpha_j^2,
\end{aligned}
\end{equation}
whereas, if $\Phi_h$ is a 2nd-order time-symmetric method, then ($Y_2 = 0$)
\begin{equation} \label{compo_o2}
  k_{1,1} = \sum_{i=1}^r \alpha_i, \qquad k_{3,1} = \sum_{i=1}^r \alpha_i^3.
\end{equation}  
}
Obviously, one also has
\[
  \widetilde{\Psi}_h = \e^{\widetilde{K}(h)}, \qquad \mbox{ with } \qquad 
  \widetilde{K}(h) = h \, \overline{k}_{1,1} M + \sum_{\ell \ge 2} h^{\ell} \, \sum_{j=1}^{c(\ell)} \overline{k}_{\ell,j} \, E_{\ell,j}
\]
and finally   
\begin{equation} \label{expreSh}
  S_h = \Psi_h \, \widetilde{\Psi}_h = \e^{K(h)} \, \e^{\widetilde{K}(h)} = \e^{M_h}.
\end{equation}

To see the requirements the modified vector field $K(h)$ has to satisfy to get an alternating-conjugate scheme of order $p$, it is illustrative to consider 
the first terms in $K(h)$ and $M_h$:
\begin{eqnarray*}
  K(h)  & = & h \, k_{1,1} M + h^2 \, k_{2,1} E_{2,1} + h^3 (k_{3,1} E_{3,1} + k_{3,2} E_{3,2}) + \mathcal{O}(h^4), \\
  M_h & = & h \, m_{1,1} M + h^2 \, m_{2,1} E_{2,1} + h^3 (m_{3,1} E_{3,1} + m_{3,2} E_{3,2}) + \mathcal{O}(h^4).
\end{eqnarray*}  
By applying the BCH formula in \eqref{expreSh} we get
\[
\begin{aligned}
   & M_h = 2 \, h \, \Re(k_{1,1}) M + 2 \, h^2 \, \Re(k_{2,1}) E_{2,1} + 2 h^3 \Big( \Re(k_{3,1}) E_{3,1} + \Re(k_{3,2}) E_{3,2} \Big) \\
   & \qquad + \frac{1}{2} h^3 (k_{1,1} \overline{k}_{2,1} - \overline{k}_{1,1} k_{2,1}) [M, E_{2,1}] + \mathcal{O}(h^4)
\end{aligned}
\]
Thus, $S_h = \Psi_h \, \widetilde{\Psi}_h$ is of order 2 if $m_{1,1} = 1, m_{2,1} = 0$ in the expression of $M_h$, so that $\Re(k_{1,1}) = \frac{1}{2}$ and $\Re(k_{2,1}) = 0$. In consequence,
\[
\begin{aligned}
 & K(h) = h \left( \frac{1}{2} + i \, d_{1,1} \right) M + i \, h^2 \, d_{2,1} E_{2,1} +  h^3 \big(k_{3,1} E_{3,1} + k_{3,2} E_{3,2} \big) + 
 \mathcal{O}(h^4) \\
 & M_h = h M  + 2 h^3 \Big( \Re(k_{3,1}) E_{3,1} + \Re(k_{3,2}) E_{3,2} \Big) - \frac{1}{2} \, i \, h^3 \, d_{2,1} [M, E_{2,1}]   + \mathcal{O}(h^4),
\end{aligned}
\]
where $d_{\ell,j} = \Im(k_{\ell,j})$.
If in addition we impose that $\Re(k_{3,1}) = \Re(k_{3,2}) = d_{2,1} = 0$, then $S_h$ is of order 3 and
\[
 K(h) = h \left( \frac{1}{2} + i \, d_{1,1} \right) M + i \, h^3  \big(d_{3,1} E_{3,1} + d_{3,2} E_{3,2} \big) + 
 \mathcal{O}(h^4).
\] 
The following proposition establishes the general structure of the modified vector fields associated
to $\Psi_h$ and $S_h$.

\begin{proposition} \label{prop53}
Let $S_h = \Psi_h \, \widetilde{\Psi}_h = \e^{K(h)} \e^{\widetilde{K}(h)} = \e^{M_h}$ be an alternating-conjugate method of order $p \ge 2$. Then $K(h)$ is such that
\[
\begin{aligned}
   & \Re(k_{1,1}) = \frac{1}{2}, \qquad \Re(k_{p,j}) = 0; \quad j=1, \ldots, c(p)\\
   & k_{\ell,j} = 0, \qquad\quad  \ell =2, \ldots, p-1; \,\, j=1, \ldots, c(\ell).
 \end{aligned}
 \]
 In other words, 
  \begin{equation} \label{structureK}
    K(h) = h \left( \frac{1}{2} + i \, d_{1,1} \right) M + i \, h^p  \sum_{j=1}^{c(p)} d_{p,j} E_{p,j} + \sum_{\ell \ge p+1} h^{\ell} \, \sum_{j=1}^{c(\ell)} k_{\ell,j} E_{\ell,j}
\end{equation}   
with $d_{i,j} \in \mathbb{R}$, $k_{\ell,j} \in \mathbb{C}$, and moreover,
\begin{equation} \label{ecu.p3}
\begin{aligned}
  & M_h = h M + h^{p+1} \left( \sum_{j=1}^{c(p+1)} 2 \, \Re(k_{p+1,j}) E_{p+1,j} - \frac{1}{2} i \sum_{j=1}^{c(p)} d_{p,j} \, [M, E_{p,j}] \right) \\
  & + \sum_{\ell = 2}^{p-1} h^{p + \ell} \, \sum_{j=1}^{c(p+\ell)} \, 2 \, \Re(k_{p+\ell,j}) \, E_{p+ \ell, j} \\
  &  + \sum_{\ell =2}^{p-1} h^{p + \ell}  \, \sum_{n=1}^{\ell -1} \, \sum_{j=1}^{c(p + \ell - n)} \, \big( \omega_{p + \ell - n, j} - i \, \nu_{p + \ell - n, j} \big) 
  \,\,  \mathrm{ad}_M^n(E_{p + \ell -n, j})  +  \mathcal{O}(h^{2p}),
\end{aligned}
\end{equation}
with $\omega_{i,j}, \nu_{i,j} \in \mathbb{R}$.
\end{proposition}

\begin{proof}
We have already seen that this is the case when $p=2$. Suppose now that $K(h)$ has the structure of eq. \eqref{structureK} for $p > 2$. 
Now, applying the BCH formula in \eqref{expreSh}, we get
\[
  M_h = K(h) + \widetilde{K}(h) + \frac{1}{2} [K(h), \widetilde{K}(h)] + \frac{1}{12} ( [K(h), [K(h), \widetilde{K}(h)]] - [\widetilde{K}(h), [K(h), \widetilde{K}(h)]] ) + \cdots.
\]
Clearly
\begin{equation} \label{2k}
  K(h) + \widetilde{K}(h) = h M + h^{p+1} \,  \sum_{j=1}^{c(p+1)} 2 \, \Re(k_{p+1,j}) E_{p+1,j} + \sum_{\ell \ge p+2} h^{\ell} \,  \sum_{j=1}^{c(\ell)} 2 \, \Re(k_{\ell,j}) E_{\ell,j}
\end{equation}
whereas
\begin{equation} \label{com1K}
 [K(h), \widetilde{K}(h)] = -i \, h^{p+1} \,  \sum_{j=1}^{c(p)} d_{p,j} \, \mathrm{ad}_M E_{p,j} - i \sum_{\ell=p+1}^{2p-2} h^{\ell +1} \sum_{j=1}^{c(\ell)} 
 \big(d_{\ell,j} - 2 d_{1,1} \Re(k_{\ell,j}) \big) \, \mathrm{ad}_M E_{\ell,j} + \mathcal{O}(h^{2p}).
\end{equation}
Here, as before, $d_{\ell,j} = \Im(k_{\ell,j})$. Successive terms in the BCH series of $M_h$ are linear combinations of nested commutators of the form $\mathrm{ad}_M^r E_{\ell,j}$ with real and imaginary coefficients.
In particular,
\[
\begin{aligned}
  & [K(h), [K(h), \widetilde{K}(h)]] =  -i \, h^{p+2} \,  \left( \frac{1}{2} + i \, d_{1,1} \right) \sum_{j=1}^{c(p)} d_{p,j} \, \mathrm{ad}_M^2 E_{p,j}  \\
  & \quad - i h \left( \frac{1}{2} + i \, d_{1,1} \right)  \sum_{\ell=p+1}^{2p-3} h^{\ell +1} \sum_{j=1}^{c(\ell)} 
 \big(d_{\ell,j} - 2 d_{1,1} \Re(k_{\ell,j}) \big) \, \mathrm{ad}_M^2 E_{\ell,j} + \mathcal{O}(h^{2p}).
\end{aligned} 
\] 
In consequence, 
the expressions  in the first and second lines of \eqref{ecu.p3} are obtained by adding the first term in \eqref{2k} and the first term in \eqref{com1K}, 
whereas adding up the remaining
terms leads to the expression in the third line.
\end{proof}

Notice that, if we additionally impose the conditions
\[
 \begin{aligned}
   & \Re(k_{p+1,j}) = 0, \quad j=1, \ldots, c(p+1) \\
   & d_{p,j} = 0, \quad j=1, \ldots, c(p),
 \end{aligned}
\]
then $S_h$ is of order $p+1$ with $K(h)$ of the form \eqref{structureK} up to order $p+1$.

One remarkable consequence of Proposition \ref{prop53} is the following.
\begin{proposition} \label{prop52}
With the hypothesis of Proposition \ref{prop53}, the alternating-conjugate method $S_h = \e^{M_h}$ can be written as
\[
  S_h = \e^{M_h} = \e^{P_h} \, \e^{W_h} \, \e^{-P_h},
\]
where
 \[
 \begin{aligned}
  & P_h =  \frac{1}{2} i \, h^p \sum_{j=1}^{c(p)} d_{p,j} E_{p,j} + 
    i \, \sum_{\ell =1}^{p-2} h^{p + \ell}  \, \sum_{n=1}^{\ell -2} \, \sum_{j=1}^{c(p + \ell - n)} \, \nu_{p + \ell - n, j} \,\,  \mathrm{ad}_M^n(E_{p + \ell -n, j}) \\
  & W_h = h M + \sum_{\ell=1}^{p-1} h^{p+ \ell}  \sum_{j=1}^{c(p+\ell)} 2 \, \Re(k_{p+\ell,j}) E_{p+\ell,j}  +
  \sum_{\ell = 2}^{p-1}  h^{p + \ell}  \, \sum_{n=1}^{\ell -1} \, \sum_{j=1}^{c(p + \ell - n)} \, \omega_{p + \ell - n, j} \,\,  \mathrm{ad}_M^n(E_{p + \ell -n, j}) \\
  & \qquad + \mathcal{O}(h^{2p}).
\end{aligned}
\]
\end{proposition}
\begin{proof}
In virtue of the well known identity $\e^A \e^B \e^{-A} = \e^C$, with $C = \e^{\mathrm{ad}_A} B$, we have
\[
  W_h =  \e^{\mathrm{ad}_{-P_h}} M_h = M_h - [P_h, M_h] + \mathcal{O}(h^{2p+2}) = M_h + [h M, P_h] + \mathcal{O}(h^{2p})
\]
so that with this choice for $P_h$, the imaginary part of $M_h$ in \eqref{ecu.p1} can be removed up to order $\mathcal{O}(h^{2p-1})$.
\end{proof}

\begin{corollary}
 If the problem \eqref{prm1} is defined in a particular Lie group $\mathcal{G}$
 and the basic scheme $\Phi_h$ in the composition \eqref{psiac} also evolves in $\mathcal{G}$, then the alternating-conjugate
 method $S_h$ of order $p$ is conjugate to a method that preserves the Lie group structure up to order $2p-1$. 
\end{corollary}
In particular, if $A_j$ are skew-Hermitian matrices, then $S_h$ is conjugate to a unitary matrix up to order $2p-1$, if $A_j$ are skew-symmetric, then $S_h$ is
conjugate to an orthogonal matrix, etc.

\

Proposition \ref{prop53} can be particularized to the case where the method $\Psi_h$ is of order $p$:
\begin{corollary} \label{prop51}
Let $\Psi_h$ be a method of order $p \ge 1$ with step-size $\frac{h}{2}$. Then, the alternating-conjugate method $S_h = \Psi_h \, \widetilde{\Psi}_h = \e^{M_h}$ is also
of order $p$ and
\begin{equation} \label{ecu.p1}
\begin{aligned}
  & M_h = h M + \sum_{\ell = 1}^{p+1} h^{p + \ell} \, \sum_{j=1}^{c(p+\ell)} \, 2 \, \Re(k_{p+\ell}) \, E_{p+ \ell, j} \\
  & \quad -i \, \sum_{\ell =2}^{p+1} h^{p + \ell}  \, \sum_{n=1}^{\ell -1} \, \sum_{j=1}^{c(p + \ell - n)} \, \nu_{p + \ell - n, j} \,\,  \mathrm{ad}_M^n(E_{p + \ell -n, j}) +
  \mathcal{O}(h^{2p+2}),
\end{aligned}
\end{equation}
where $\nu_{i,j} \in \mathbb{R}$.
\end{corollary}  

\begin{proof}
One has to impose $d_{1,1}=0$ and $d_{p,j}=0$ in eq. \eqref{structureK}. Then, coefficients $\omega_{i,j}$ vanish, because all
the successive commutators $[K(h), \ldots, [K(h),\widetilde{K}(h)]\ldots]$ appearing in the BCH formula applied to $\e^{K(h)} \, \e^{\widetilde{K}(h)}$
contain only pure
imaginary terms up to order $\mathcal{O}(h^{2p+1})$.
\end{proof}
In consequence, if the problem \eqref{prm1} is defined in a Lie group $\mathcal{G}$, then the scheme $S_h = \Psi_h \, \widetilde{\Psi}_h$ of order $p$ is conjugate to a method that preserves the
Lie group structure up to order $2p+1$. This is so up to order $2p+2$ if $\Psi_h$ is symmetric-conjugate and the order $p$ is even, due
to the particular structure of $K_h$ in that case \cite{blanes22osc}. On the other hand, if $\Psi_h$ is symmetric-conjugate and its order $p$ is odd, 
then the resulting AC method is time-symmetric and of order $p+1$.

\subsection{New alternating-conjugate methods}
\label{sub.5.2}

The analysis of the previous subsection shows that, in addition to concatenating a given method $\Psi_h$ of order $p$ (say palindromic or symmetric-conjugate)
with the same scheme with complex conjugate coefficients, one can also get an alternating-conjugate method of order $p \ge 2$ by considering $\Psi_h$ as in Proposition
\ref{prop53}, namely by requiring the following order conditions:
\begin{equation} \label{orconac}
\begin{aligned}
   & \Re(k_{1,1}) = \frac{1}{2}, \qquad \Re(k_{p,j}) = 0; \quad j=1, \ldots, c(p)\\
   & k_{\ell,j} = 0, \qquad\quad  \ell =2, \ldots, p-1; \,\, j=1, \ldots, c(\ell),
 \end{aligned}
 \end{equation} 
 {where $k_{i,j}$ are as in \eqref{expreKh} and depend explicitly on the coefficients $\alpha_j$ of the composition \eqref{psiac}.}
 
The simplest AC method of order $p=2$ corresponds of course to the composition $S_h = \Phi_{\alpha h} \, \Phi_{\overline{\alpha} h}$ with $\Phi_h$ such that 
\[
 \Phi_{\alpha h} = \e^{K(h)} \qquad \mbox{ and } \qquad K(h) = \alpha h M + \alpha^2 h^2 Y_2 + \alpha^3 h^3 Y_3 + \cdots.
\]
By imposing $\Re(\alpha) = \frac{1}{2}, \Re(\alpha^2) = 0$ we get $\alpha = \frac{1}{2} (1\pm i )$, i.e., we recover method \eqref{slt1}. 

Analogously, for an AC method of order 3 within this family one has to take
$\Psi_h = \Phi_{\alpha_1 h} \, \Phi_{\alpha_2 h} \, \Phi_{\alpha_3 h}$ to satisfy the 5 required conditions \eqref{orconac}. {Taking into account eq. \eqref{compo_o1}, 
these are explicitly
\[
\begin{aligned}
  &  \Re(\alpha_1 + \alpha_2 + \alpha_3) = 1/2, \qquad  \Re( \alpha_1^3 + \alpha_2^3 + \alpha_3^3) = 0 \\
  & \Re \big( -\alpha_1^2 (\alpha_2 + \alpha_3) - \alpha_2^2 \alpha_3 + \alpha_1 (\alpha_2^2 + \alpha_3^2) + \alpha_2 \alpha_3^2 \big) = 0 \\
  & \Re( \alpha_1^2 + \alpha_2^2 + \alpha_3^2) = 0, \qquad \Im( \alpha_1^2 + \alpha_2^2 + \alpha_3^2) = 0.
\end{aligned}  
\]
}
Although there are solutions with $\alpha_3 \in \mathbb{R}$,
it is more efficient to consider $\Phi_h$ in the composition \eqref{psiac} as a 2nd-order time-symmetric method, namely
\begin{equation} \label{psiac2}
  \Psi_h = \Psi_{\alpha_1 h}^{[2]} \, \Psi_{\alpha_2 h}^{[2]} \, \cdots \, \Psi_{\alpha_r h}^{[2]}.
\end{equation}
Now the number of order conditions \eqref{orconac} to achieve a method of order 3, 4, 5, 6 is, respectively, 2, 4, 7 and 11. This is the strategy we follow next to construct higher order schemes with the minimum number of basic methods (or \emph{stages}).
We denote, for brevity, the whole AC method by its
sequence of coefficients:
\[
   S_h = (\alpha_1, \alpha_2, \ldots, \alpha_r, \overline{\alpha}_1, \overline{\alpha}_2, \ldots, \overline{\alpha}_r).
\]

\paragraph{Order 3.} The two order conditions, namely $ \Re(k_{1,1}) = \frac{1}{2}, \Re(k_{3,1})=0$ {given explicitly by \eqref{compo_o2},
can be already satisfied by taking $r=1$, i.e., 
\[
   \Re(\alpha_1) = \frac{1}{2},  \quad \Re(\alpha_1^3) = 0, \qquad \mbox{ leading to } \qquad \alpha_1 = \frac{1}{2} \pm i \frac{1}{2 \sqrt{3}}.
\]
}
In this way we recover the scheme \eqref{bc3}, which is both symmetric-conjugate and AC. Notice that if $\Psi_h^{[2]}$ is taken as the Strang splitting \eqref{strang} 
for two operators, then the number of exponentials is 5 instead of 12 with a composition of the Lie--Trotter scheme.

\paragraph{Order 4.} The 4 order conditions for the scheme $\Psi_h$, namely
\[
{ \Re(k_{1,1}) = \frac{1}{2}, \qquad \Re(k_{3,1}) = 0, \qquad \Im(k_{3,1}) = 0, \qquad \Re(k_{4,1}) = 0}
\]    
require taking $\Psi_h$ in \eqref{psiac2} as
\begin{equation} \label{nac4}
  \Psi_h = \Psi_{\alpha_1 h}^{[2]} \,  \Psi_{\alpha_2 h}^{[2]}
\end{equation}
and the only solutions are given by
\[
  (\alpha_1 = a, \, \alpha_2 = i \, \overline{a}), \quad \mbox{ and }  \quad (\alpha_1 =i \, \overline{a}, \, \alpha_2 = a), 
  \quad \mbox{ with } \quad 
  a = \frac{1}{4} \left(1 + \frac{1}{\sqrt{3}}\right) + i \, \frac{1}{4} \left(1 - \frac{1}{\sqrt{3}}\right)
\]
(together with the complex conjugate, of course). Notice that the resulting $S_h = (\alpha_1, \alpha_2, \overline{\alpha}_1, \overline{\alpha}_2)$ 
has the same number of stages as the scheme $S_{h}^{(1)}$ in \eqref{ac.1}. {This is the scheme depicted in Figure \ref{figure.5} as a green dashed line.}

\paragraph{Order 5.} {Now there are 7 order conditions that is required to solve: 
\[
  \Re(k_{1,1}) = \frac{1}{2}, \qquad \Re(k_{3,1}) = \Im(k_{3,1}) = \Re(k_{4,1}) = \Im(k_{4,1}) = 0, \qquad \Re(k_{5,1}) = \Re(k_{5,2}) = 0,
\]
so that we need four parameters, $\alpha_1, \ldots, \alpha_4$, or equivalently, a composition of the form  
\[
   \Psi_h = \Psi_{\alpha_1 h}^{[2]} \, \Psi_{\alpha_2 h}^{[2]} \, \Psi_{\alpha_3 h}^{[2]} \, \Psi_{\alpha_4 h}^{[2]}
\]
to satisfy all of them.} In fact, we can take one
parameter, say $\alpha_1$, as a real number. {In this case, one has to apply numerical techniques to solve the equations. Specifically, we use 
the function \texttt{fsolve} of the Python SciPy library, based on the classical Powell method \cite{virtanen20sci}.

After an exhaustive search, 16 distinct solutions (including their complex conjugates) that satisfy the imposed symmetry have been found. Among them, we have selected the one for which the sum of the absolute values of its coefficients is minimal:}
\[
\begin{array}{l}
  \alpha_1 = 0.13073364974455472155, \\ 
   \alpha_2 = 0.10154067971150062704 +  0.13578392847671735429 \, i, \\
  \alpha_3 =  0.16195992616393787750 -   0.05016739165848310348 \, i, \\
  \alpha_4 =  0.10576574438000677391 +  0.07684331129821891226 \, i.
\end{array}  
\]
In this way, the final AC composition reads
\[
  S_h = (\alpha_1, \alpha_2, \alpha_3, \alpha_4, \alpha_1, \overline{\alpha}_2,  \overline{\alpha}_3,  \overline{\alpha}_4).
\]  

\paragraph{Order 6.} A method $\Psi_h$ verifying the required 11 order conditions involves at least 6 stages, although one of the coefficients can be taken as
a real number. {By applying the same procedure as in the preceding case, we have found 101 distinct solutions, along with their complex conjugates, and the one with minimal value of the 
sum of the absolute values of its coefficients is}
\[
\begin{array}{l}
  \alpha_1 = 0.051834036182240306862, \\ 
   \alpha_2 = 0.075584762328805037429 + 0.068952097954972525370 \, i, \\
  \alpha_3 =  0.126191199798221549793 - 0.022451017530352466819 \, i, \\
  \alpha_4 =  0.067883683573696296147 - 0.098039677222465976320 \, i, \\
  \alpha_5 = 0.099243916328147654969 + 0.049312230362166446543 \, i, \\
  \alpha_6 = 0.079262401788889154800 - 0.041953102069126791785 \, i.
\end{array}  
\]
The final AC method reads
\begin{equation} \label{nac6}
  S_h = (\alpha_1, \ldots, \alpha_6, \overline{\alpha}_1, \ldots,  \overline{\alpha}_6).
\end{equation}

\medskip

Since there are several ways to construct AC compositions, it is relevant at this point to compare them based on the number of stages involved in each
type of composition, {which is determined by the number of order conditions. We should notice, however, that an actual method may require more stages than
order conditions, as shown by the previous 5th- and 6th-order alternating-conjugate schemes. This is also the case if the order conditions do not have solutions
with the minimum number of parameters.}

{We collect in Table \ref{table.1} the  results up to order 8 for different types of compositions of a 2nd-order time-symmetric method.
In particular, $\mathrm{P}$ refers to palindromic compositions (hence only even orders are considered), and the collected numbers have been obtained a number
of times in the literature (see e.g. \cite{blanes24smf}).  $\mathrm{SC}$ denote symmetric-conjugate compositions of a 2nd-order time-symmetric basic scheme,
already analyzed in \cite{blanes22osc}. The symbols 
$\mathrm{P}$--$\widetilde{\mathrm{P}}$ and $\mathrm{SC}$--$\widetilde{\mathrm{SC}}$ refer to alternating-conjugate methods obtained taking $\Psi_h$
as a palindromic and a symmetric-conjugate composition, respectively. Notice that both of them are palindromic (and hence time-symmetric) by construction. This
fact also contributes to reduce the number in the $\mathrm{SC}$--$\widetilde{\mathrm{SC}}$ case. Finally, the column denoted as $\mathrm{AC}$ is obtained
by counting the number of order conditions \eqref{orconac} rounded to the next even number.} 

It is clear that, taking only into consideration the computational cost required by
each type of method, $\mathrm{SC}$--$\widetilde{\mathrm{SC}}$ compositions are preferred for orders 6 and 8. To get efficient methods within this class one
must then build good SC schemes of orders 5 and 7 involving the minimum number of stages. The coefficients of two such integrators are collected in Table \ref{table.2}, {and they have been obtained by applying the same methodology as for the previous 5th- and 6th-order AC methods.}

\begin{table}[ht]
\begin{center}
    \renewcommand\arraystretch{1.1}
\begin{tabular}{|c|c|c||c|c|c|}\hline
 \multicolumn{6}{|c|}{\bfseries Minimum number of stages}\\ \hline
 Order & $\mathrm{P}$ & $\mathrm{SC}$ & $\mathrm{P}$--$\widetilde{\mathrm{P}}$ & $\mathrm{SC}$--$\widetilde{\mathrm{SC}}$  & $\mathrm{AC}$ \\ \hline
  3  &   & 2 &  &  &  2 \\
  4 & 3 & 3 & 6 & 4 & 4  \\
  5 &  & 5 &  & & 8  \\
  6 & 7 & 7 & 14 & 10 & 12  \\ 
  7 &  & 11 &  &  & 18  \\ 
  8 & 15 & 16 & 30 & 22 & 26  \\ \hline
\end{tabular}  
\end{center}
\caption{\label{table.1} \small Minimum number of stages required by each type of composition of a 2nd order time-symmetric basic method to achieve a given order. $\mathrm{P}$
stands for palindromic; $\mathrm{SC}$ for symmetric-conjugate; $\mathrm{P}$--$\widetilde{\mathrm{P}}$ refers to an alternating-conjugate method obtained
by concatenating 2 palindromic compositions with complex conjugate coefficients; $\mathrm{SC}$--$\widetilde{\mathrm{SC}}$ is an alternating-conjugate
method obtained from symmetric-conjugate compositions, and $\mathrm{AC}$ are direct alternating-conjugate methods. Note that $\mathrm{P}$--$\widetilde{\mathrm{P}}$, $\mathrm{SC}$--$\widetilde{\mathrm{SC}}$ and $\mathrm{AC}$ are all alternating-conjugate schemes.}
\end{table}

\begin{table}[ht]
  \begin{center}
    \renewcommand\arraystretch{1.1}
    \begin{tabular}{lll}
      \multicolumn{3}{c}{$\Psi_{h,sc}^{[5]} = \Psi_{\alpha_1 h}^{[2]} \, \Psi_{\alpha_2 h}^{[2]} \, \Psi_{\alpha_3 h}^{[2]} \, \Psi_{\overline{\alpha}_2 h}^{[2]} \, 
         \Psi_{\overline{\alpha}_1 h}^{[2]}$}\\
      \hline 
     $\Re( \alpha_1) = 0.17526840907207411405$ & \quad $\Im(\alpha_1) = 0.05761474413053870201$ \\
      $\Re( \alpha_2) = 0.18487368019298416043$ & \quad $\Im(\alpha_2) = - 0.19412192275724958851$ \\
      $\alpha_3 = 0.27971582146988345102$ &  \\
        \hline 
                                    & & \\                                    
      \multicolumn{3}{c}{$\Psi_{h,sc}^{[7]} = \Psi_{\alpha_1 h}^{[2]} \, \cdots \Psi_{\alpha_5 h}^{[2]} \, \Psi_{\alpha_6 h}^{[2]} \, 
      \Psi_{\overline{\alpha}_5 h}^{[2]} \, \cdots \, \Psi_{\overline{\alpha}_1 h}^{[2]}$}\\
      \hline 
      $\Re( \alpha_1) = 0.05211820743645156337$ & \quad $\Im(\alpha_1) = -0.05814624289751311388$ \\
      $\Re( \alpha_2) = 0.10923197827620526541$ & \quad $\Im(\alpha_2) = 0.02935068872383690377$ \\
      $\Re( \alpha_3) = 0.09943629453321852209$ & \quad $\Im(\alpha_3) = -0.06231578289901792940$ \\
      $\Re( \alpha_4) = 0.08136441998830503070$ & \quad $\Im(\alpha_4) = 0.11683729387729571634$ \\
      $\Re( \alpha_5) = 0.14644914726793223517$ & \quad $\Im(\alpha_5) = 0.04299436701496493366$ \\
      $\alpha_6 = 0.02279990499577476650$ &  \\
      \hline 
    \end{tabular}
   \end{center}        
  \caption{\label{table.2} 
  \small{Coefficients of symmetric-conjugate methods of order 5 and 7 obtained as compositions of a basic 2nd order time-symmetric basic scheme with
  the minimum number of stages. They can be used to construct alternating-conjugate methods of order 6 and 8, respectively.}}
\end{table}

\subsection{Numerical examples}
\label{sec.5.3}

While in the previous sections we have analyzed the main qualitative properties of alternating-conjugate numerical schemes, our purpose here is
to illustrate the efficiency of the new integrators on practical computations involving some of the linear differential systems considered before.

\paragraph{Linear Hamiltonian systems.} 
We first consider a Hamiltonian system \eqref{Ham.1} described by a matrix $H$ of the form \eqref{eq:conj_int} with 5 frequencies $\omega_j > 0$ chosen as
random numbers, $0.2077 < \omega_j < 0.8443$, and split it into two parts $H = A + B$, {as in the previous examples}. 
We also take a random initial condition $y_0 \in \mathbb{R}^{10}$, integrate the system until the final time $t_f = 60$
and compute the error in the preservation of the energy $\mathcal{H}(y) = y^T H y /2$ at the final time for different step sizes $h$. Finally, we depict this
error as a function of the number of basic {2nd-order} maps $\Psi_h^{[2]}$ evaluated by each numerical integrator. 
In this way we get the graph of Figure \ref{figure.6} (left). {Lines coded as AC-4 and AC-6 correspond to AC methods \eqref{nac4} and \eqref{nac6}, of order 4 and 6, respectively, whereas lines denoted as P-4, P-6 and P-8} are obtained by the most efficient 4th-, 6th-
and 8th-order
palindromic compositions with real coefficients involving, respectively, 5, 13 and 21 symmetric 2nd-order maps $\Psi_h^{[2]}$ per integration step (see \cite{blanes24smf}, Table 8.1). 

Notice that, whereas the slope of each curve clearly exhibits the order of the corresponding method, there are significant differences between schemes with 
real and complex coefficients with respect to stability, with the former one having a smaller threshold value $h^*$ for the step size, so that for $h > h^*$ the methods
are unstable. In this sense, the behavior of the new AC scheme \eqref{nac6} is particularly remarkable.
Different values for the frequencies and initial conditions provide of course different results, but the observed pattern is quite similar. We should mention, however,
that the extra cost due to the use of complex arithmetic is not taken into account in this figure. Alternating-conjugate methods of order 6 and 8 obtained
by concatenating the schemes of Table \ref{table.2} provide less efficient results in all cases.

\begin{figure}[!ht] 
\centering
  \includegraphics[width=.49\textwidth]{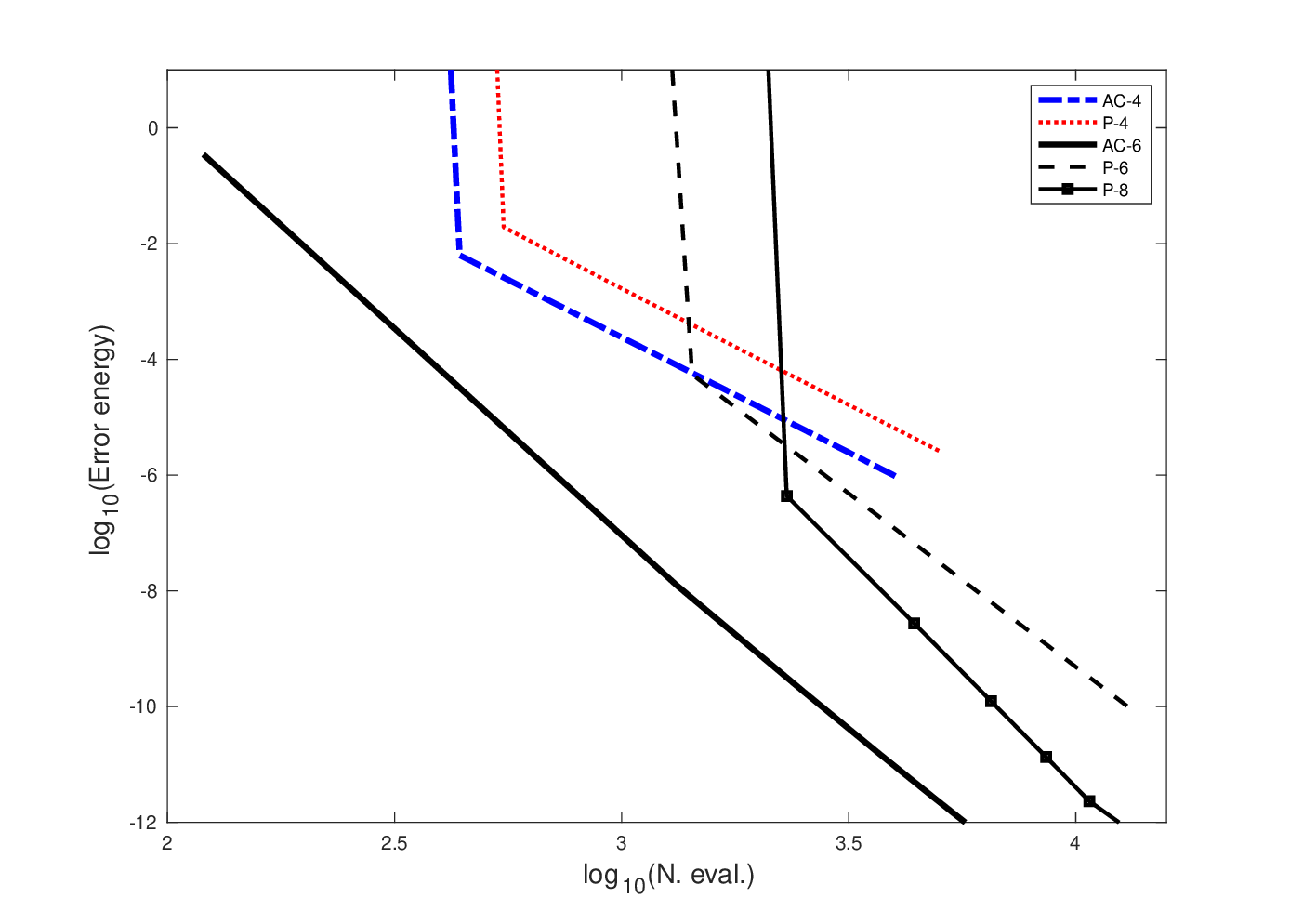}
  \includegraphics[width=.49\textwidth]{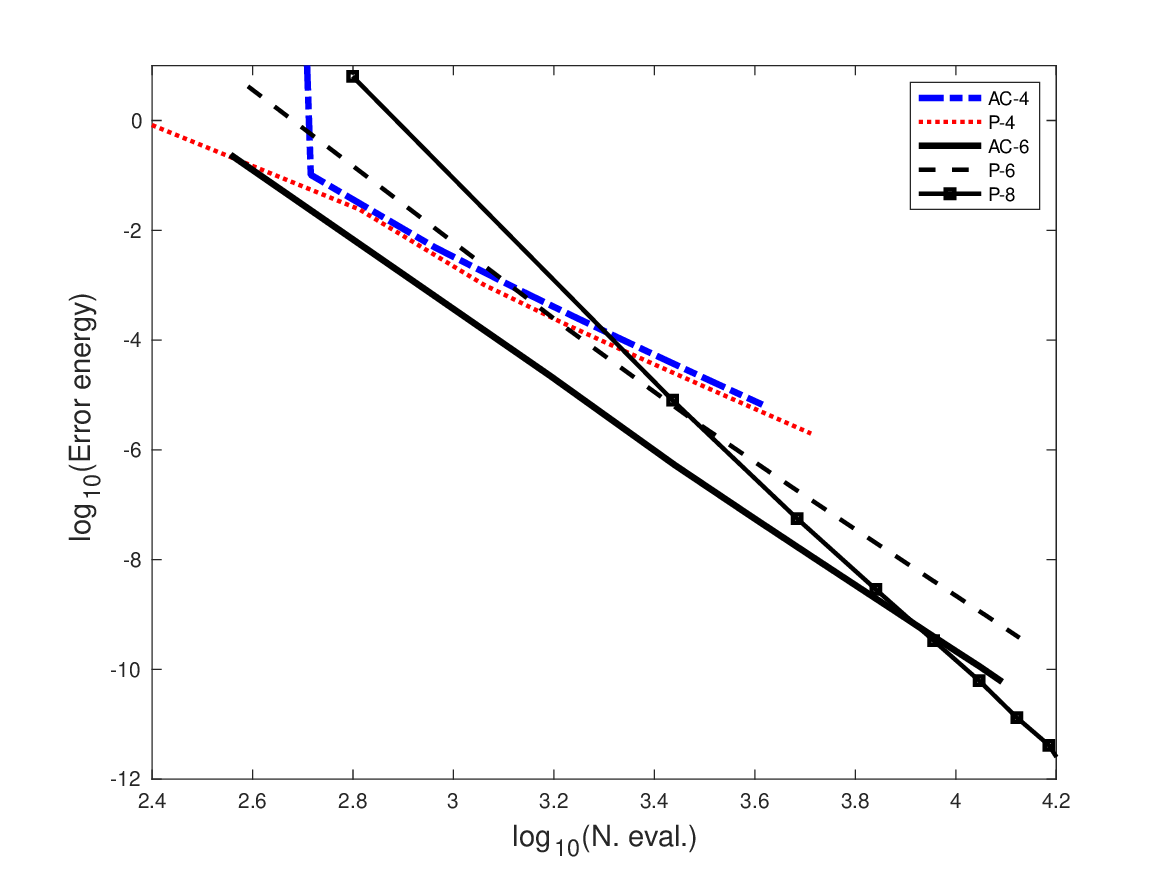}
\caption{\label{figure.6} \small Error in energy vs. number of evaluations of the basic Strang splitting $\Psi_h^{[2]}$ for a linear real Hamiltonian system (left) and a linear
unitary flow (right) of dimension $N=10$. {Lines AC-4 and AC-6} correspond to the new AC methods \eqref{nac4} and \eqref{nac6}, of order 4 and 6. 
{Lines P-4, P-6 and P-8}
correspond to palindromic compositions with real coefficients of order 4, 6 and 8, recommended in \cite{blanes24smf} for their superior efficiency.}
\end{figure}

\paragraph{Unitary flows.}

We check next the same methods in the linear unitary system 
\[
     \dot{u} = i H u, \qquad u(0) = u_0,
\]     
where $H$ is taken as a $10 \times 10$ complex Hermitian matrix and $u_0 \in \mathbb{C}^{10}$, both chosen at random. Again, we split $H$ into two parts and take the
Strang splitting as the basic scheme. We integrate until a final time
$t_f = 40$ and check the preservation of the energy $\mathcal{H}(u) = \overline{u}^T H u$ as before. The results are depicted in Figure \ref{figure.6} (right),
{where the same notation has been used for the methods}.
Notice that in this case, all methods require using complex arithmetic, so there is no overcost for the new alternating-conjugate schemes. Although, as we have
seen before, stability is an issue in this case when using complex coefficients, in practice the threshold value $h^*$ is quite large, specially for method \eqref{nac6},
which shows the best efficiency for all values of $h$ tested. Observe that it also compares favorably with the most efficient standard 8th-order scheme with
real coefficients on a wide range of values of $h$.

\section{Concluding remarks}
\label{sec.6}

In this work we have shown that, {concatenating a given splitting method involving complex coefficients $\alpha_j$ with the same scheme but with the complex conjugate
coefficients $\overline{\alpha}_j$}, results in a new (alternating-conjugate) scheme with good long time behavior and favorable preservation properties when applied to several classes of linear problems. 
This analysis shows, in addition, that it is possible to get alternating-conjugate composition methods of order $p$ directly, without requiring 
concatenating methods of the same order $p$. Some of the new schemes thus constructed exhibit good efficiency for some problems in comparison with standard composition methods
with real coefficients.

{
Although only linear problems have been considered here, one may wonder what happens when this class of composition schemes is applied for the numerical
integration of nonlinear systems of equations. There are several considerations to be made here. First, the vector field defining
the system of differential equations has to be analytic, at least in the part of the complex plane containing the path where the integration is carried out. Otherwise,
order reductions take place unless a proper implementation is adopted, as e.g. in \cite{casas21cop} for the semi-linear complex Ginzburg--Landau equation. Second,
the analysis of section \ref{sub.5.1} is still valid when the nonlinear problem is formulated in terms of the Lie formalism. Thus, as long as the problem is defined in a certain Lie group (for instance, a Hamiltonian system), 
an AC scheme of order $p$ is conjugate to a method that preserves the Lie group structure up to order $2p+1$, and thus one still may expect preservation of
geometric structures up to higher orders than the order of the method itself. 

Let us focus on a (nonlinear) Hamiltonian system $H(q,p)$. When using a numerical method with complex coefficients, the resulting symplectic map corresponds
(via backward error analysis \cite{hairer06gni}) to a modified equation that remains Hamiltonian, though now governed by a complex Hamiltonian function that is formally preserved. Numerical experiments demonstrate that the original Hamiltonian $H(q,p)$
 is still well preserved over time. However, to obtain rigorous results on the long-term preservation of first integrals, it is crucial that the imaginary parts of both coordinates and momenta remain sufficiently small. Otherwise, significant instabilities may arise.
 
 Although the numerical flow generated by the AC scheme is conjugate to a real Hamiltonian system up to order $r > p$
 this does not guarantee that $(\Im(q), \Im(p))$ remains small over timescales of order $h^{-r}$. In fact, numerical results reveal a clear correlation between the magnitude of  $(\Im(q), \Im(p))$ and the preservation of invariants. Moreover, the specific region of phase space from which initial conditions are selected also significantly influences the observed behavior. Overall, these complexities raise several important issues that merit further investigation, in any case
 well beyond the scope of this paper.
 }

Even when one has only two operators, $m=2$ in the linear equation \eqref{prm1}, deserves a separated analysis, given its potential applications, in particular  to the
time-dependent Schr\"odinger equation. In that situation, the modified vector field associated with
the map \eqref{gen.split} is of the form \eqref{expreKh}, with the same dimensions $c(\ell)$ involved for $\ell \ge 2$, whereas $c(1) = 2$. In consequence, the analysis
of subsection \ref{sub.5.1} is valid for the corresponding alternating-conjugate schemes with minor modifications. 
It might be the case, however, that the new integrators obtained by considering the splitting
\eqref{gen.split}  differ from those collected in subsection \ref{sub.5.2} in terms of efficiency. This issue will be analyzed elsewhere.

\subsection*{Acknowledgements}
The work of JB is supported by ANR-22-CE40-0016 ``KEN" of the Agence Nationale de la Recherche (France) and by the region Pays de la 
Loire (France) through the project ``MasCan". SB, FC and AE-T acknowledge financial support by 
Ministerio de Ciencia e Innovaci\'on (Spain) through project  PID2022-136585NB-C21, 
MCIN/AEI/10.13039/501100011033/FEDER, UE.

\subsection*{Compliance with Ethical Standards}

All authors declare that they have no conflicts of interest.

\appendix

\section{Appendix: multiple eigenvalues}
In this appendix we state an extension of  Theorem \ref{thm:unitary} allowing the matrix $H$ to have repeated eigenvalues at the cost of an extra spectral assumption on the leading error term. 
The arguments to get this generalization from Theorem \ref{thm:unitary} are exactly the same as those presented in \cite{bernier23scs} to pass from Theorem 3.1 to Theorem 3.3, so we omit the proof.
\begin{theorem}
\label{thm:refin}
Let  $H \in \mathbb{C}^{N\times N}$ be a complex matrix and let $S_\ts \in \mathbb{C}^{N\times N}$ be a family of invertible complex matrices depending smoothly on $\ts\in \mathbb{R}$ such that
\begin{itemize}
\item the spectrum of $S_h$ satisfies $\sigma(S_\ts) =\overline{\mathrm{inv}(\sigma(S_\ts))}$
\item $S_\ts$ is consistent with $\exp(i\ts H)$, i.e. 
$$
\displaystyle S_\ts \mathop{=}_{\ts\to 0} \exp \big(  i\ts H  + i \ts^{p+1}  R+ \mathcal{O}(\ts^{p+q+1}) \big),
$$
where $p\geq 1$ is the consistency order, $q\geq 1$ and $R$ is a complex matrix,
\item  $H$ is diagonalizable and its eigenvalues are real,
\item For all $\omega \in \sigma(H)$, the eigenvalues of $\Pi_\omega R_{| E_\omega(H)}$ are real and simple, where $\Pi_\omega$ denotes the spectral projector onto $E_\omega (H):= \mathrm{Ker}(H-\omega I_N) $.
\end{itemize}
Then there exist 
\begin{itemize}
\item $D_\ts$,  a family of real diagonal matrices depending smoothly on $\ts$,
\item $P_\ts$, a family of complex invertible matrices depending smoothly on $\ts$,
\end{itemize}
such that $P_\ts = P_0 + \mathcal{O}(\ts^q)$, $D_\ts = D_0 + \mathcal{O}(\ts^q)$, both $P_0^{-1} B P_0$ and $P_0^{-1} H P_0$ are diagonal, where $B := \sum_{\omega \in \sigma(H)} \Pi_\omega \, R \, \Pi_\omega$  and, provided that $|\ts|$ is small enough, 
\begin{equation}
\label{eq:what_we_want_refin}
S_\ts = P_\ts \, \e^{i\ts D_\ts} \, P_\ts^{-1}.
\end{equation}
\end{theorem}

\begin{remark}
The extra spectral assumption on $R$ is equivalent to ask that, provided that $|\ts|$ is small enough, the eigenvalues of $H+h^p R$ are real and simple.
That this is typically the case for alternating-conjugate methods follows readily from the analysis of subsection \ref{sub.5.1}, and explains the numerical results
collected in section \ref{sec.2}.
\end{remark}

By the same arguments as \cite[Corollary 3.4.]{bernier23scs}, we deduce, as previously, a corollary about the preservation of the super-actions, the mass and the energy. 
\begin{corollary} \label{cor:refin}
 In the setting of Theorem \ref{thm:refin}, if moreover $H$ and $R$ are Hermitian, there exists a constant $C>0$ such that, provided that $|\ts|$ is small enough,  for all $u\in \mathbb{C}^N$ and all $\omega \in \sigma(H)$, one gets
$$
 \sup_{n\geq 0} \, \Big||\Pi_\omega S_\ts^{n} u|^2 - |\Pi_\omega  u|^2 \Big| \leq C |\ts|^q |u|^2,
$$
and the mass and the energy are almost conserved, i.e. for all $u\in \mathbb{C}^N$, it holds that
\[
  \sup_{n \in \mathbb{Z}} \,  \big| \mathcal{M}(S_\ts^n u) - \mathcal{M}( u)  \big|\leq C |\ts|^q |u|^2  \qquad \mathrm{ and } \qquad 
  \sup_{n \in \mathbb{Z}} \, \big| \mathcal{H}(S_\ts^n u) - \mathcal{H}( u)  \big| \leq C |\ts|^q |u|^2, 
\]
where, as before, $\mathcal{M}( u) = |u|^2 $ and $\mathcal{H}( u) = \overline{u}^T H u $.
\end{corollary}


\bibliographystyle{siam}

\end{document}